\documentclass[12pt]{article}

\usepackage{amsmath,amsfonts,amssymb}
\usepackage{bbm}
\newtheorem{defn}{Definition}[section]
\newtheorem{theo}[defn]{Theorem}
\newtheorem{lem}[defn]{Lemma}
\newtheorem{prop}[defn]{Proposition}
\newtheorem{cor}[defn]{Corollary}
\newtheorem{rem}[defn]{Remark}
\newtheorem{exam}[defn]{Example}
\newenvironment{proof}{{\bf Proof }}{{\vskip 0.1cm \hfill$\Box$}}

\begin{document} 

\noindent
{\Large\bf On countably skewed Brownian motion with accumulation point}

\bigskip
\noindent
{\bf Youssef Ouknine}
\footnote{The research of Youssef Ouknine was supported by the \lq\lq Mathematics and Applications\rq\rq\ Project of Hassan II Academy of Sciences and Technology.},
{\bf Francesco Russo}
{
\footnote{The research of Francesco Russo was partially supported by the \lq\lq FMJH Program Gaspard Monge in optimization and operation research\rq\rq\ (Project 2014-1607H). 
}, 
{\bf Gerald Trutnau}
{\footnote{The research of Gerald Trutnau was supported by Basic Science Research Program through the National Research Foundation of Korea(NRF) funded by the Ministry of Education, Science and Technology(MEST)(2013028029) and Seoul National University Research Grant 0450-20110022.}
}}\\

\noindent
{\small{\bf Abstract.}  In this work we connect the theory of symmetric Dirichlet forms and direct stochastic calculus to obtain strong existence and pathwise 
uniqueness for Brownian motion that is perturbed by a series of constant multiples of local times at a sequence of points that has exactly one accumulation point in $\mathbb{R}$. The considered process is identified as special 
distorted Brownian motion $X$ in dimension one and is studied thoroughly.  
Besides strong uniqueness, we present necessary and sufficient conditions for non-explosion, recurrence and positive recurrence as well as for $X$ to be semimartingale and possible applications to advection-diffusion in layered media.}\\ 

\noindent
{Mathematics Subject Classification (2010): primary; 31C25, 60J60, 60J55; secondary: 31C15, 60B10.}\\

\noindent 
{Key words: Skew Brownian motion, local time, strong existence, pathwise uniqueness, transience, recurrence, positive recurrence.}

\section{Introduction}\label{one}
In this paper we are concerned with a special distorted Brownian motion in dimension one. 
Distorted BM in dimension $d$ was first introduced in \cite{ahkrs}. It is roughly speaking the Hunt process associated to the regular Dirichlet form
$$
{\cal E}(f,g):=\frac{1}{2}\int_{\mathbb{R}^d} \nabla f\cdot \nabla g \,d\mu, \ \ \ \ f,g\in D({\cal E}),
$$
on $L^2(\mathbb{R}^d;\mu)$, where $\mu$ is a Radon measure on $\mathbb{R}^d$ with full support. If $\mu=\rho\,dx$ is absolutely continuous with respect to the Lebesgue measure $dx$, then  the conditions on $\rho$ for an extension of 
$({\cal E}, C_0^{\infty}(\mathbb{R}^d))\subset D({\cal E}))$ to be associated to a Hunt process are quite weak 
(cf. \cite[Theorem 3.1.6]{fot}). In particular, if the partial derivatives of $\rho$ are absolutely continuous and sufficiently regular, we obtain using integration by parts 
$$
{\cal E}(f,g)=-\int_{\mathbb{R}^d} \left (\frac{1}{2}\Delta f +\frac{\nabla \rho}{2\rho}\cdot \nabla f\right ) g \,d\mu,
$$
and we can see that the process associated to $({\cal E},D({\cal E}))$ is a $d$-dimensional BM with drift $\frac{\nabla \rho}{2\rho}$ (cf. \cite{fuku81}).\\
In this work, we will take a particularly probabilistic viewpoint on distorted BM. For a given a.e. strictly positive nice function $\rho$, distorted BM in dimension one 
with initial condition $x\in \mathbb{R}$ 
may be regarded (whenever it makes sense) as a solution to 
\begin{eqnarray}\label{dist0}
X_t=x + W_t + \int_{\mathbb{R}} \ell_t^a(X)\frac{d\rho(a)}{2\rho(a)},
\end{eqnarray}
where $W$ is a standard BM and $\ell^a(X)$ the symmetric semimartingale local time of $X$ at $a\in \mathbb{R}$. 
(\ref{dist0}) makes in particular sense, when $\rho$ is weakly 
differentiable with derivative $\rho'\in L^1_{loc}(\mathbb{R};dx)$, and $\frac{1}{\rho}$ is not too singular. In this case, we obtain by the occupation times formula  
$$
\int_{\mathbb{R}} \ell_t^a(X)\frac{d\rho(a)}{2\rho(a)}=\int_0^t\frac{\rho'}{2\rho}(X_s)ds,
$$ 
so that (\ref{dist0}) is a BM with logarithmic derivative as a drift. The Bessel processes of dimension $\delta\in (1,2)$ fall into this category with $\rho(x)=|x|^{\delta-1}$. In this paper, however, we will consider a very special $\rho$ whose logarithmic derivative has no absolutely continuous component. More precisely, 
we consider a concrete simple function $\rho$ whose logarithmic derivative is an infinite sum of Dirac measures such that (\ref{dist0}) can be rewritten as
\begin{eqnarray}\label{dist}
X_t=x+W_t+\sum_{k\in \mathbb{Z}}(2\alpha_k-1)\ell_t^{z_k}(X), 
\end{eqnarray}
where $(z_k)_{k\in \mathbb{Z}}$ is an unbounded sequence of real numbers that may have an accumulation point and $\alpha_k\in (0,1)$, $k\in \mathbb{Z}$ are real numbers (see (\ref{distbm(i)}) for the concrete $\rho$).\\ 
To our knowledge, an equation of the form (\ref{dist}) first occurs explicitly in \cite{taka86}, \cite{taka86a} as special one dimensional case. 
There, weak existence and pathwise uniqueness of some multidimensional analogue of (\ref{dist}) with additional diffusion coefficient and absolutely continuous drift was studied. However, \cite{taka86}, \cite{taka86a} do not allow for accumulation points and the one dimensional case is already covered by earlier work of Le Gall \cite{lg84}.
There Le Gall obtained weak existence and pathwise uniqueness of (\ref{dist0}) in a general global setting, 
where $\frac{d\rho(a)}{2\rho(a)}$ is replaced by an arbitrary signed measure $\nu(da)$ with globally bounded total variation and whose absolute value on atoms is strictly less than one (in Le Gall's setting there is also some diffusion coefficient, see (\ref{legall0}) below). Although Le Gall's global condition applied to (\ref{dist}) is equivalent to $\sum_{k\in \mathbb{Z}}|2\alpha_k-1|<\infty$ and is quite strong, it does not exclude the possibility of accumulation points. In case 
$(z_k)_{k\in \mathbb{Z}}$ has no accumulation points and $\sum_{k\in \mathbb{Z}}|2\alpha_k-1|<\infty$, Ramirez considers in  \cite{ram} the pathwise unique solution to (\ref{dist}) of \cite{lg84} as regular diffusion (cf. \cite{ka02}) and presents interesting applications to advection-diffusions in layered media. Our work includes all the mentioned cases. For a more detailed discussion, we refer to the end of Section \ref{3.1}.\\
We discover at least two interesting phenomena that seem to be generic for equations with a drift as in (\ref{dist}) and we fully characterize these. First, a solution to  (\ref{dist}), 
which by definition is a semimartingale and continuous up to infinity (hence non-explosive) may exist, even if $\sum_{\{k\,|\, z_k \in U_0\}}|2\alpha_k-1|=\infty$, where $U_0$ is any neighborhood of the accumulation point (see Remark \ref{counterex} and Example \ref{Bessel}). In fact the semimartingale property is equivalent to $\rho(a)$ being locally of bounded variation (see (\ref{semi1}) below).  
In particular, we are able to consider (\ref{dist0}), even if $\frac{d\rho(a)}{\rho(a)}$ is not locally of bounded total variation. Furthermore, (\ref{dist}) is not automatically non-explosive, i.e. 
a solution to (\ref{dist}) might not exist. We present an example with explosion in finite time where the sequence $(z_k)_{k\in \mathbb{Z}}$ has an accumulation point at infinity, but none in $\mathbb{R}$ (see Example  \ref{counterexample}). In this case, we may nonetheless consider a solution up to lifetime with local times defined in the Dirichlet form sense via the Revuz correspondence (cf. Remark \ref{localtime}). \\
A pathwise unique solution to (\ref{dist}) shall be called countably skewed Brownian motion in order to contrast with the terminology of multi-skewed Brownian motion in \cite{ram} 
when $(z_k)_{k\in \mathbb{Z}}$ has no accumulation point.
For the proof of strong existence and pathwise uniqueness of (\ref{dist}), we need two types of conditions. The first one is the local condition $\sum_{\{k\,|\, z_k \in U_0\}}|2\alpha_k-1|<\infty$ and the second is a 
global condition that ensures non-explosion (see Theorem \ref{legallweaker1} and discussions in Remarks \ref{leGallweaker3} and \ref{conservativeglobalprop} where we relate our work to \cite{lg84}).  
Both conditions are explicit. The local condition is optimal in the sense that it is equivalent to the existence of a nice scale function (see Remark \ref{hbddvarrem}) and the global condition on non-explosion is sharp (see Proposition \ref{conservativenessdir}, Corollary \ref{conservativeness2} and Remark \ref{conservativeglobalprop}). We emphasize that the global conditions (\ref{global1}) and (\ref{global2}) are directly readable from the density $\rho$ in (\ref{gammadef}) of the underlying Dirichlet form $({\cal E}, D({\cal E}))$ determined by (\ref{closureDF}). 
In fact, the construction of a solution to (\ref{dist}) is performed via Dirichlet form theory. The key point is to identify (\ref{dist}) as distorted BM. 
This is done in Proposition \ref{legallweaker0} where starting from (\ref{dist}), the density $\rho$ for which (\ref{dist}) is a distorted BM with respect to 
the Dirichlet form given by (\ref{closureDF}) is determined. The identification of the distorted BM (\ref{weak1}) in Theorem \ref{semimart} with a solution to  equation (\ref{dist}) (see (\ref{weak3}) in Corollary \ref{countablyskewbm}) is done with the help of (\ref{localtime(ii)}). Note that the approach through distorted BM, i.e. through the process associated to the Dirichlet form (\ref{closureDF}), is more general than the approach through (\ref{dist}), since distorted BM does not need to be semimartingale. 
Necessary and sufficient conditions for the latter are presented in Theorem \ref{semimart}. \\
In addition to the above mentioned results, we present necessary and sufficient conditions for transience, recurrence and positive recurrence, as well as a sufficient condition for the existence of a unique invariant distribution (see Theorem \ref{recurrence}, Theorem \ref{recurrence2} and Corollary \ref{uniqueinvariantdist}). These results are quite standard from the existence of a scale function $h$ as in Remark \ref{hbddvarrem} and similar results were also presented in \cite{ram}. However, we insist on explicitly pointing out that in each of these statements, additionally to the statements corresponding to the scale function, an equivalent condition for the Dirichlet form (\ref{closureDF}) is presented. This 
underlines our bidirectional approach. \\
In section \ref{33}, we use the theory of generalized Dirichlet forms as  applied in \cite{St1}, as well as the results of this work to propose a generalization for the longitudinal and transversal directions of advection-diffusion in layered media considered in \cite{ram} and \cite{ram2} 
(see Remark \ref{advection}). \\
Finally, we want to say a few words on skew reflected diffusions and corresponding uniqueness results. If all $\alpha_k$ except $\alpha:=\alpha_1$ are $\frac12$ in (\ref{dist}) and $z_1=0$, then $X$ is called the $\alpha$-skew Brownian motion (see Remark \ref{skewbm}). It was first considered by It\^o and McKean  (see e.g. \cite[Section 4.2, Problem 1]{itomckean}) and strong uniqueness was derived in \cite{hs}. Skew reflected diffusions and strong uniqueness results have been considered by many authors then. Additionally to \cite{taka86, taka86a, lg84} the existence and uniqueness results of \cite{es}, \cite{rut}, \cite{rut90} and \cite{bc2} are particularly close to ours. A survey on skew reflected diffusions is given in \cite{lejay}.

\section{Construction and basic properties of a countably skew reflected Brownian motion}\label{two}
In this section, we first construct a countably skew reflected Brownian motion by Dirichlet form methods. As a byproduct of the construction 
method, its basic properties like diffusion and semimartingale property as well as the explicit form of SDE that it solves are directly readable 
from the density $\rho$ of the Dirichlet form. Besides in remarks and examples, we point out some remarkable features of the constructed process.\\
 We consider two sequences of real numbers $(l_k)_{k\in \mathbb{Z}}$ and  $(r_k)_{k\in \mathbb{Z}}$ such that
$$
l_k< l_{k+1}<0 < r_k < r_{k+1}, \ \ \ \forall k\in \mathbb{Z}, 
$$
$$
\mbox{and} \ \lim_{k\to \infty}l_k = 0 =\lim_{k\to -\infty}r_k.
$$
We suppose further, that zero is the sole accumulation point of the sequences $(l_k)_{k\in \mathbb{Z}}$, $(r_k)_{k\in \mathbb{Z}}$. In particular 
$$
\lim_{k\to -\infty}l_k=-\infty \ \mbox{ and } \lim_{k\to \infty}r_k=\infty.
$$
Let $(\gamma_k)_{k\in \mathbb{Z}}$ and  $(\overline{\gamma}_k)_{k\in \mathbb{Z}}$ be another two sequences of arbitrary, but strictly positive real numbers. Let
\begin{eqnarray}\label{gammadef}
\rho(x):=\sum_{k\in \mathbb{Z}}\left \{\gamma_{k+1}1_{(l_k,l_{k+1})}+\overline{\gamma}_{k+1}1_{(r_k,r_{k+1})} \right \}(x),
\end{eqnarray}
where $1_A$ is the indicator function of the set $A$ and $(a,b)$ the open interval from $a$ to $b$. Since $\rho$ appears as density to the Lebesgue measure we do not have to care about the values of $\rho$ at the boundary points $l_k, r_k, 0$ at the moment. 
However, whenever we have to choose a pointwise version, we will choose its \lq\lq symmetric\rq\rq\ version $\tilde{\rho}$ (cf. (\ref{distbm(i)}) below).
With the sole exception of Proposition \ref{legallweaker0}, we always assume that
$$
\rho\in L^1_{loc}(\mathbb{R};dx).
$$
Then $\rho dx$ is a positive Radon measure and the bilinear form 
\begin{eqnarray}\label{closureDF}
{\cal E}(f,g):=\frac{1}{2}\int_{\mathbb{R}}f'(x)g'(x)\rho(x)dx, \ \ \ f,g\in C_0^{\infty}(\mathbb{R}),
\end{eqnarray}
is well defined. Here  $C_0^{\infty}(\mathbb{R})$ denotes the space of infinitely often continuously differentiable functions with compact support 
and $f'$ denotes the derivative of $f$. By definition of the sequences $(\gamma_k)_{k\in \mathbb{Z}}$ and  $(\overline{\gamma}_k)_{k\in \mathbb{Z}}$, $\rho$ 
is bounded above and below away from zero on each compact subset of $\mathbb{R}\setminus \{0\}$. Thus each point of $\mathbb{R}\setminus \{0\}$ is a regular point for $\rho$, 
and so by the results of \cite[Theorem 3.1.6]{fot} $({\cal E}, C_0^{\infty}(\mathbb{R}))$ 
is closable in $L^2(\mathbb{R};\rho dx)$. The closure $({\cal E}, D({\cal E}))$ is a regular symmetric Dirichlet form (see \cite[pages 3-6]{fot}). 
Indeed, the regularity, i.e. that $C_0(\mathbb{R})\cap D({\cal E})$ is dense both in the continuous functions with compact support $C_0(\mathbb{R})$ 
and in $D({\cal E})$, just follows from the fact that we constructed $({\cal E}, D({\cal E}))$ as the closure of $({\cal E}, C_0^{\infty}(\mathbb{R}))$. 
The submarkovian property of $({\cal E}, D({\cal E}))$ (called Markovian property in \cite{fot}) follows easily by showing \cite[(1.1.6)]{fot} with 
the help of a mollifier as in \cite[Problem 1.2.1]{fot}. \\
Let $(T_t)_{t\ge 0}$ be the strongly continuous submarkovian (called Markovian in \cite{fot}) contraction semigroup on $L^2(\mathbb{R};\rho dx)$ that is associated to $({\cal E}, D({\cal E}))$ (see \cite[Section 1.3]{fot}). 
By general Dirichlet form theory (see \cite[Chapter 7]{fot}) there exists a Hunt process with life time $\zeta$ and cemetery $\Delta$
$$
((X_t)_{t\ge 0}, {\cal F}, ({\cal F}_t)_{t\ge 0}, \zeta, (P_x)_{x\in \mathbb{R}\cup \{\Delta\}})
$$ 
such that $x\mapsto E_x[f(X_t)]$ is a quasi-continuous $dx$-version of $T_t f$ for any (Borel measurable) $f\in L^2(\mathbb{R};\rho dx)$, and $E_x$ denotes the expectation with respect to $P_x$. \\ \\
The semigroup $(T_t)_{t\ge 0}$ can be regarded as a semigroup on $L^{\infty}(\mathbb{R};\rho dx)$ (see \cite[p. 49]{fot}). 
Then $({\cal E}, D({\cal E}))$ is called {\it conservative}, if $T_t 1_{\mathbb{R}}(x)=1$ for $dx$-a.e. $x\in \mathbb{R}$ and any $t\ge 0$. 
{\bf Throughout this section} (but see Remark \ref{localtime}), {\bf we assume} 
\begin{itemize}
	\item[(H0)] $({\cal E}, D({\cal E}))$ is conservative.
\end{itemize}
For instance, if 
$$
\int_{1}^{\infty} \frac{r}{\log v(r)}dr=+\infty, 
$$
where $v(r):=\int_{B_r(0)}\rho(x)dx$, or if there exists some $T>0$  such that for any $R>0$ 
$$
\liminf_{r\to \infty}\frac{e^{-\frac{r^2}{2T}}}{r}\int_{\{|x|<\sqrt{e^{R+r}-1}\}}\rho(x)dx = 0,
$$
then $({\cal E}, D({\cal E}))$ is conservative (see \cite[Theorem 4]{stu} and \cite[Theorem 2.2]{Ta89}). 
Necessary and sufficient conditions are presented a posteriori in Proposition \ref{conservativenessdir}, Corollary \ref{conservativeness2} and Remark \ref{conservativeglobalprop}, in case (\ref{semi1}) and (\ref{hbddvargamma}) below hold.\\ \\
Let $\mbox{cap}$ be the capacity related to $({\cal E}, D({\cal E}))$ as defined in \cite[p.64]{fot}. Since $C_0^{\infty}(\mathbb{R})$ is a special 
standard core for $({\cal E}, D({\cal E}))$, and $({\cal E}, D({\cal E}))$ is  except in zero locally comparable with the Dirichlet form $\frac{1}{2}\int_{\mathbb{R}}f'(x)g'(x) dx$, $f,g\in H^{1,2}(\mathbb{R}):=\{f\in L^2(\mathbb{R};dx)\,|\, f'\in L^2(\mathbb{R}; dx)\}$ of Brownian motion, it follows from \cite[Lemma 2.2.7 (ii), and Theorem 4.4.3]{fot} that 
$$
\mbox{cap}(\{x\})>0\ \mbox{ for any }\ x\in \mathbb{R}\setminus\{0\}. 
$$
We will consider the following assumption on $({\cal E}, D({\cal E}))$:
\begin{itemize}
	\item[(H1)] $\mbox{cap}(\{0\})>0$.
\end{itemize}
\begin{rem}\label{caph1}
(H1) holds if for instance $\exists\lim_{k\to\infty}\gamma_k,\lim_{k\to-\infty}\overline{\gamma}_k>0$. In this case $\rho$ is locally bounded away from zero and above. Therefore, the Dirichlet norm of $({\cal E}, D({\cal E}))$ is  
(everywhere) locally comparable with the Dirichlet norm of $\frac{1}{2}\int_{\mathbb{R}}f'(x)g'(x) dx$ on $H^{1,2}(\mathbb{R})$, which is as we remarked before the Dirichlet form of Brownian motion. $({\cal E}, D({\cal E}))$ has hence the same exceptional sets as Brownian motion, i.e. (H1) holds (see \cite[Lemma 2.2.7(ii) and Theorem 4.4.3]{fot}). 
\end{rem}
\begin{prop}\label{consdiff}
Under (H1), the Hunt process $((X_t)_{t\ge 0}, {\cal F}, ({\cal F}_t)_{t\ge 0}, \zeta, (P_x)_{x\in \mathbb{R}\cup \{\Delta\}})$ associated to 
$({\cal E}, D({\cal E}))$ is a conservative diffusion, i.e. we have:
\begin{itemize}
	\item[(i)] The process has infinite life time, namely
$$
P_x[\zeta=\infty]=P_x[X_t\in \mathbb{R}, \forall t\ge 0]=1\ \mbox{ for all }\ x\in \mathbb{R}.
$$
\item[(ii)] The process is a diffusion, namely
$$
P_x[t\mapsto X_t\ \mbox{ is continuous on }\ [0,\infty)]=1\ \mbox{ for all }\ x\in \mathbb{R}.
$$
\end{itemize}
If (H1) does not hold, then in general (i), (ii) are only valid for all $x\in \mathbb{R}\setminus\{0\}$.
\end{prop}
\begin{proof}
We have $P_{\cdot}[X_t\in \mathbb{R}]=T_t 1_{\mathbb{R}}$ $dx$-a.e. Thus by (H0) $R_1 1_{\mathbb{R}}(x):=\int_0^{\infty}e^{-t}P_x[X_t\in \mathbb{R}]dt=1$ for $dx$-a.e. $x\in\mathbb{R}$. Since 
$R_1 1_{\mathbb{R}}$ is $1$-excessive, it follows that $R_1 1_{\mathbb{R}}(x)=1$ for all $x\in \mathbb{R}\setminus N_1$, where 
$N_1\subset\mathbb{R}$ satisfies $\mbox{cap}(N_1)=0$. It follows $P_x[\zeta=\infty]=P_x[X_t\in \mathbb{R}, \forall t\ge 0]=1$  for all $x\in \mathbb{R}\setminus N_1$.\\ 
Since $({\cal E}, D({\cal E}))$ is {\it (strongly) local} we obtain by \cite[Theorem 4.5.1 (ii)]{fot} that 
$$
P_x[t\mapsto X_t\ \mbox{ is continuous on }\ [0,\zeta)]=1\ \mbox{ for all }\ x\in \mathbb{R}\setminus N_2
$$
where $N_2\subset\mathbb{R}$ satisfies $\mbox{cap}(N_2)=0$. By considering $N:=N_1\cup N_2$ if necessary, we may assume that $N:=N_1=N_2$. Since $\{0\}$ is the only non-trivial subset of $\mathbb{R}$ which might have zero capacity we obtain that $((X_t)_{t\ge 0}, (P_x)_{x\in \mathbb{R}})$ is a conservative diffusion for any 
$x\in \mathbb{R}\setminus\{0\}$, and under (H1) for any $x\in \mathbb{R}$.
\end{proof}\\ \\
Next, we want to identify the stochastic differential equation verified by $(X_t)_{t\ge 0}$. In Dirichlet form theory this is done via the Fukushima decomposition 
for $(X_t)_{t\ge 0}$ in the following way:\\
Let $(L,D(L))$ denote the self-adjoint generator of $({\cal E}, D({\cal E}))$ (cf. \cite[Section 1.3]{fot}). Then
\begin{eqnarray}\label{id2afirst}
-{\cal E}(f,g)= \int_{\mathbb{R}} Lf\cdot g \,\rho\, dx, \ f\in D(L), g\in D({\cal E}).
\end{eqnarray}
Now, in order to identify the drift of $(X_t)_{t\ge 0}$, we have to evaluate $L$ at the identity map which is (typically) 
even not locally in $D(L)$. 
However, the identity map is (typically) locally in $D({\cal E})$ and so the left hand side of (\ref{id2afirst}) can be evaluated. 
The drift is then identified with a signed Radon measure (cf. (\ref{id2a})). If this measure is smooth in the sense of \cite{fot} 
then it corresponds uniquely to a continuous additive functional via the Revuz correspondence (cf. \cite[Theorem 5.1.4]{fot}). 
This additive functional it then the drift part of $(X_t)_{t\ge 0}$. The identification of the local martingale part in Fukushima's 
decomposition is particularly easy in our situation. Since its quadratic variation 
is related to the energy $\cal{E}$ one obtains
$$
(\langle M^{[id]} \rangle_t)_{t\ge 0} \leftrightarrow (id')^2 \rho dx=\rho dx
$$
and the Revuz measure of $A_t \equiv t$ is $\rho dx$ (by a straightforward calculation). So by uniqueness of the Revuz correspondence $\langle M^{[id]} \rangle_t\equiv t$.
\begin{theo}\label{semimart}
(i) Under (H1), the family $\mathbb{M}:=((X_t)_{t\ge 0}, (P_x)_{x\in \mathbb{R}})$ associated to 
$({\cal E}, D({\cal E}))$ satisfies: for $x\in \mathbb{R}$, $((X_t)_{t\ge 0}, P_x)$ is a semimartingale, if and only if
\begin{equation}\label{semi1}\tag{S0}
\sum_{k\le 0}|\overline{\gamma}_{k+1}-\overline{\gamma}_{k}|+\sum_{k\ge 0}|\gamma_{k+1}-\gamma_{k}|<\infty.
\end{equation}
In particular (\ref{semi1}) implies 
\begin{eqnarray}\label{semi2}
\exists \lim_{k\to\infty}\gamma_k=:\gamma\in [0,\infty) \ \mbox{ and }\ \exists \lim_{k\to-\infty}\overline{\gamma}_k=:\overline{\gamma}\in [0,\infty).
\end{eqnarray}
(ii) If (H1) does not hold, then $\mathbb{M}$ is (always) a semimartingale for all $x\in \mathbb{R}\setminus \{0\}$.\\
(iii) Suppose (H1) and (\ref{semi1}) hold. Then we have for all $x\in \mathbb{R}$  
\begin{eqnarray}\label{weak1}
X_t=x+W_t+ \sum_{k\in \mathbb{Z}}\left \{\frac{\gamma_{k+1}-\gamma_{k}}{2}\ell_t^{l_k}+\frac{\overline{\gamma}_{k+1}-\overline{\gamma}_{k}}{2}\ell_t^{r_k}\right \}+
\frac{\overline{\gamma}-\gamma}{2}\ell^0_t, \ t\ge 0,\ P_x\mbox{-a.s.}
\end{eqnarray}
where $(\ell^a_t)_{t \ge 0}$ is the unique positive continuous additive functional (PCAF) of $\mathbb{M}$ (cf. \cite[Chapter 5.1]{fot}) 
that is associated via the Revuz correspondence (cf. \cite[Theorem 5.1.3]{fot}) to the smooth measure 
$\delta_a$, $a\in \mathbb{R}$, and $((W_t)_{t\ge 0}, ({\cal F}_t)_{t\ge 0}, P_x)$ is a Brownian motion starting from zero for all $x\in \mathbb{R}$.\\
(iv) If (H1) does not hold, then 
\begin{eqnarray}\label{weak2}
X_t=x+W_t+ \sum_{k\in \mathbb{Z}}\left \{\frac{\gamma_{k+1}-\gamma_{k}}{2}\ell_t^{l_k}+\frac{\overline{\gamma}_{k+1}-\overline{\gamma}_{k}}{2}\ell_t^{r_k}\right \}, \ t\ge 0,\ P_x\mbox{-a.s.}
\end{eqnarray}
for all $x\in \mathbb{R}\setminus \{0\}$.
\end{theo}
\begin{proof}
(i) Let $id(x):=x$ for $x\in\mathbb{R}$. Then $id\in D({\cal E})_{loc}$ (cf \cite[p. 117]{fot} for the definition), and for any $v\in C_0^{\infty}(\mathbb{R})$ we calculate
\begin{eqnarray}\label{id1}
-{\cal E}(id,v) & = & \lim_{k\to\infty}\left \{-\frac{1}{2}\int_{-\infty}^{l_{k+1}}v'(x)\rho(x)dx-\frac{1}{2}\int^{\infty}_{r_{-k}}v'(x)\rho(x)dx\right \} \nonumber \\ 
& = &  \lim_{n\to\infty} \{\sum_{k\le n}\frac{\gamma_{k+1}-\gamma_{k}}{2}\int_{\mathbb{R}}v(x) \delta_{l_k}(dx) -\frac{\gamma_{n+1}}{2}\int_{\mathbb{R}}v(x)\delta_{l_{n+1}}(dx) \nonumber \\
&& + \sum_{k\ge -n}\frac{\overline{\gamma}_{k+2}-\overline{\gamma}_{k+1}}{2}\int_{\mathbb{R}}v(x) \delta_{r_{k+1}}(dx) +\frac{\overline{\gamma}_{-n+1}}{2}\int_{\mathbb{R}}v(x)\delta_{r_{-n}}(dx) \}  \\ \nonumber
\end{eqnarray}
where $\delta_x$ denotes the Dirac measure in $x\in\mathbb{R}$. Thus in an informal way, we can write
\begin{eqnarray}\label{id2a}
-{\cal E}(id,v) =\int_{\mathbb{R}}v(x)\nu(dx),
\end{eqnarray}
with
\begin{eqnarray}\label{id2}
\nu=\sum_{k\in \mathbb{Z}}\left \{\frac{\gamma_{k+1}-\gamma_{k}}{2}\delta_{l_k}+\frac{\overline{\gamma}_{k+1}-\overline{\gamma}_{k}}{2}\delta_{r_k} \right \}+
\frac{\lim_{k\to-\infty}\overline{\gamma}_k-\lim_{k\to\infty}\gamma_k}{2}\delta_0,
\end{eqnarray}
Under (H1) the notion of smooth measure is equivalent to the notion of Radon measure, 
i.e. a positive measure on $\mathbb{R}$ is smooth, if and only if it is locally finite in $\mathbb{R}$. 
Thus $\nu$ in (\ref{id2}) is a signed smooth measure, if and only if its positive and negative parts are locally finite in $\mathbb{R}$.
The last is the case, if and only if (\ref{semi1}) holds, because (\ref{semi1}) clearly implies (\ref{semi2}).  
Now the statement follows easily by \cite[Theorem 5.5.4]{fot}.\\
(ii) If (H1) does not hold, then $\nu$ is a signed smooth measure, if and only if it is locally finite in $\mathbb{R}\setminus \{0\}$. 
But $\nu$ is always locally finite in $\mathbb{R}\setminus \{0\}$ and so the assertion follows from \cite[Theorem 5.5.4]{fot}. \\
(iii) By \cite[Theorem 5.5.1]{fot} we only have to calculate the local martingale part $M^{[id]}_t$, and the local zero energy part $N^{[id]}_t$ 
appearing in the local Fukushima decomposition for $A^{[id]}_t=X_t-X_0=M^{[id]}_t+N^{[id]}_t$. By (\ref{id1}), (\ref{semi1}), and (\ref{semi2}) it immediately follows with 
\cite[Corollary 5.5.1]{fot} that 
\begin{eqnarray}\label{zero}
N_t^{[id]}=\sum_{k\in \mathbb{Z}}\left \{\frac{\gamma_{k+1}-\gamma_{k}}{2}\ell_t^{l_k}+\frac{\overline{\gamma}_{k+1}-\overline{\gamma}_{k}}{2}\ell_t^{r_k}\right \}+
\frac{\overline{\gamma}-\gamma}{2}\ell^0_t, t\ge 0.
\end{eqnarray}
Under (H1) the equality in (\ref{zero}) is strict, i.e. it holds $P_x$-a.s. for all $x\in\mathbb{R}$. 
Since  $M^{[id]}_t$ is a continuous local martingale it suffices to show that for its quadratic variation, 
we have $\langle M^{[id]} \rangle_t=t$. The Revuz measure $\mu_{\langle M^{[id]} \rangle}$ of $\langle M^{[id]} \rangle$ satisfies 
$$
\mu_{\langle M^{[id]} \rangle}=\rho dx
$$
which is the same than the Revuz measure of the additive functional $A_t=t$. Thus the equality $\langle M^{[id]} \rangle_t=t$  is strict by  
(H1).\\
(iv) If (H1) does not hold, then using (ii) and the same line of arguments as in (iii), 
with $u:=id$ and test functions $v\in D({\cal E})_{b, F_n}$, $n\ge 1$ in \cite[Theorem 5.5.4]{fot}, where $(F_n)_{n\ge 1}$ 
is a generalized nest, we obtain that
\begin{eqnarray}\label{zero2}
N_t^{[id]}=\sum_{k\in \mathbb{Z}}\left \{\frac{\gamma_{k+1}-\gamma_{k}}{2}\ell_t^{l_k}+\frac{\overline{\gamma}_{k+1}-\overline{\gamma}_{k}}{2}\ell_t^{r_k}\right \}, t\ge 0.
\end{eqnarray}
$P_x$-a.s. for all $x\in\mathbb{R}\setminus \{0\}$. As in (iii) we obtain $\langle M^{[id]} \rangle_t=t$ in the sense of equivalence of PCAFs. 
Thus $\langle M^{[id]} \rangle_t=t$  $P_x$-a.s. for all $x\in\mathbb{R}\setminus \{0\}$. This completes our proof.
\end{proof}
\begin{rem}\label{localtime}
If we do not assume (H0), then we obtain Theorem \ref{semimart} exactly as before, 
except that the semimartingale property and the identification of the associated process only hold up to the lifetime $\zeta$, i.e. for $t<\zeta$. 
Indeed, the corresponding process is then a diffusion up to lifetime, i.e. 
Proposition \ref{consdiff}(ii) holds with $\infty$ replaced by $\zeta$ (see proof of Proposition \ref{consdiff}) and the semimartingale 
property, as well as the identification of the process can be worked out up to lifetime exactly as in the proof of Theorem \ref{semimart}.\\
\end{rem}
The PCAFs $(\ell_t^a)_{t\ge 0}$ in Theorem \ref{semimart} can be uniquely determined up to a constant. 
If $((X_t)_{t\ge 0}, P_x)$ is a semimartingale, then
\begin{eqnarray}\label{localtime(ii)}
\ell^{l_k}=\frac{2}{\gamma_{k+1}+\gamma_k}\ell^{l_k}(X), \ \ \ell^{r_k}=\frac{2}{\overline{\gamma}_{k+1}+\overline{\gamma}_{k}}\ell^{r_k}(X), \ \ \ell^{0}=\frac{2}{\overline{\gamma}+\gamma}\ell^{0}(X)
\end{eqnarray}
$P_x$-a.s. for any $k\in\mathbb{Z}$ and for $\overline{\gamma}+\gamma\not= 0$, where $\ell^a(X)$, $a\in\mathbb{R}$, denotes the symmetric 
semimartingale local time at $a$ of $((X_t)_{t\ge 0}, P_x)$  as defined in \cite[VI. (1.25) Exercise]{ry05}. Once the process $X_t$ is a semimartingale, 
this can be carried out by comparing the symmetric Tanaka formula (see \cite[VI. (1.2) Theorem]{ry05} for the left version of it) for $|X_t-a|$ 
with the local Fukushima decomposition (cf. \cite[Theorem 5.5.1]{fot}) for 
$|X_t-a|$, where we choose $a=0$ for $\overline{\gamma}+\gamma\not= 0$, and $a=l_k,\, r_k$, $k\in \mathbb{Z}$. 
This is done in all details for the point $a=0$ in the introduction of \cite{RuTr4}, but the procedure is exactly the same for any other point. So, we 
omit the proof. Therefore, the following corollary follows immediately from Theorem \ref{semimart}.
\begin{cor}\label{countablyskewbm}
\begin{itemize}
\item[(i)] Suppose that $((X_t)_{t\ge 0}, (P_x)_{x\in \mathbb{R}})$ is a semimartingale. If (H1) holds, then for any $x\in \mathbb{R}$
\begin{eqnarray}\label{weak3}
X_t=x+W_t+ N_t, \ \ t\ge 0,\ P_x\mbox{-a.s.},
\end{eqnarray}
with
\begin{eqnarray}\label{n1}
N_t=\sum_{k\in \mathbb{Z}}\left \{(2\alpha_k -1)\ell_t^{l_k}(X)+(2\overline{\alpha}_k-1)\ell_t^{r_k}(X)\right \}+(2\alpha-1)\ell_t^0(X),
\end{eqnarray}
where $(\ell^a_t(X))_{t \ge 0}$  is the symmetric semimartingale local time of $((X_t)_{t\ge 0}, (P_x)_{x\in\mathbb{R}})$ at $a$, and (with $\gamma, \overline{\gamma}$ as defined in (\ref{semi2}))
$$
\alpha=\frac{\overline{\gamma}}{\overline{\gamma}+\gamma}\ \mbox{ if } \ \overline{\gamma}\not=\gamma \ \mbox{ otherwise }\  \alpha=\frac{1}{2}, \ \  \alpha_k=\frac{\gamma_{k+1}}{\gamma_{k+1}+\gamma_k}, \ \ 
\overline{\alpha}_k=\frac{\overline{\gamma}_{k+1}}{\overline{\gamma}_{k+1}+\overline{\gamma}_{k}},\ \ k\in \mathbb{Z}.
$$
\item[(ii)] If (H1) does not hold, then (\ref{weak3}) holds for any $x\in \mathbb{R}\setminus \{0\}$ with
\begin{eqnarray}\label{n2}
N_t=\sum_{k\in \mathbb{Z}}\left \{(2\alpha_k -1)\ell_t^{l_k}(X)+(2\overline{\alpha}_k-1)\ell_t^{r_k}(X)\right \},
\end{eqnarray}
where $(\ell^a_t(X))_{t \ge 0}$  is the symmetric semimartingale local time of $((X_t)_{t\ge 0}, (P_x)_{x\in\mathbb{R}\setminus\{0\}})$ at $a$.
\end{itemize}
\medskip
Consequently, we have  $P_x$-a.s. $\langle X\rangle_t = \langle W\rangle_t=t$ for any $t\ge 0$ and for all $x\in \mathbb{R}$ in case of (i) (resp. for all $x\in \mathbb{R}\setminus \{0\}$ in case of (ii)). 
Thus by the occupation times formula \cite{ry05}, we have $\int_0^t 1_{\{y\}}(X_s)ds=\int_{\mathbb{R}}1_{\{y\}}(a)\ell^a_t(X)da=0$ $P_x$-a.s. for any $x,y\in \mathbb{R}$ in case of (i) 
(for any $x,y\in \mathbb{R}\setminus\{0\}$ in case of (ii)), and so also
\begin{eqnarray}\label{n3} 
\int_0^t 1_{\{y\}}(X_s)dW_s=0,
\end{eqnarray}
$P_x$-a.s. for any $t\ge 0$ and $x,y\in \mathbb{R}$ in case of (i) (for any $x,y\in \mathbb{R}\setminus\{0\}$ in case of (ii)).\\
\end{cor}
Let us choose a \lq\lq symmetric\rq\rq\ pointwise version of $\rho$ 
\begin{eqnarray}\label{distbm(i)}
\tilde{\rho}&:=&\sum_{k\in \mathbb{Z}}\left \{\gamma_{k+1}1_{(l_k,l_{k+1})}+\overline{\gamma}_{k+1}1_{(r_k,r_{k+1})}
+\frac{\gamma_{k+1}+\gamma_{k}}{2}1_{\{l_k\}}+\frac{\overline{\gamma}_{k+1}+\overline{\gamma}_{k}}{2}1_{\{r_k\}} \right \} \nonumber \\
&& +\frac{\overline{\gamma}+\gamma}{2}1_{\{0\}}.\\ \nonumber
\end{eqnarray}
Condition (\ref{semi1}) implies that $\tilde{\rho}$ is locally of bounded variation and so $d\tilde{\rho}$ is a signed Radon measure that is 
locally of bounded total variation. In particular, it can be written as
$$
d\tilde{\rho}(a)=\sum_{k\in \mathbb{Z}}\left \{(\gamma_{k+1}-\gamma_{k})\delta_{l_k}(da)
+(\overline{\gamma}_{k+1}-\overline{\gamma}_{k})\delta_{r_k}(da) \right \}+(\overline{\gamma}-\gamma)\delta_{0}(da).
$$
Then clearly $N_t$ in (\ref{n1}) equals
$$
\frac{1}{2}\int_{\mathbb{R}}\ell_t^a(X)\frac{d\tilde{\rho}(a)}{\tilde{\rho}(a)}.
$$
and so (\ref{weak3}) has the form
\begin{eqnarray}\label{distgeneral0}
X_t=x + W_t + \frac{1}{2}\int_{\mathbb{R}} \ell_t^a(X)\frac{d\tilde{\rho}(a)}{\tilde{\rho}(a)}.
\end{eqnarray}
We will see below in Remark \ref{counterex} and Example \ref{Bessel} that the signed measure 
$\frac{d\tilde{\rho}(a)}{\tilde{\rho}(a)}$ needs not to be locally of bounded total variation in general.\\

\begin{rem}\label{distbm}
Let  $\tilde{\rho}$ be as in (\ref{distbm(i)}). Instead of (\ref{distgeneral0}), we could have considered the more general equation 
\begin{eqnarray}\label{distgeneral1}
X_t=x + \int_0^t \sigma(X_s)dW_s +\int_0^t b(X_s)ds+ \frac{1}{2}\int_{\mathbb{R}} \ell_t^a(X)\frac{d\tilde{\rho}(a)}{\tilde{\rho}(a)}.
\end{eqnarray}
Indeed, this is possible for very general $\sigma$ and $b$ by considering instead of the bilinear form (\ref{closureDF}) on $L^2(\mathbbm{R},\rho dx)$, the bilinear form 
\begin{eqnarray*}
{\cal E}(f,g):=\frac{1}{2}\int_{\mathbb{R}}\sigma^2(x)f'(x)g'(x)\rho(x)\varphi(x)dx, \ \ \ f,g\in C_0^{\infty}(\mathbb{R})
\end{eqnarray*}
on $L^2(\mathbbm{R},\rho \varphi dx)$, where
$$
\varphi(x):=\frac{1}{\sigma(x)^{2}}e^{\int_0^x \frac{2b}{\sigma^2}(y)dy}.
$$
If $\sigma$ and $b$ are not too singular, (\ref{distgeneral1}) may be derived by similar techniques as presented here for $\varphi\equiv 1$. 
Thus, we do not expect any new phenomena resulting from $\sigma$ and $b$,
 except if $\sigma$ and $b$ are very singular as for instance in
 \cite{FRT1, FRT2, RTAP}. Such an analysis however, mixing the techniques of  \cite{FRT1, FRT2, RTAP} and the countably skew reflected
Brownian motion framework would lead us
too far and is more suitably performed in a subsequent work.\\
\end{rem}
\begin{rem}\label{counterex}
Assume that (H1) holds. It may then happen that (\ref{semi1}) holds, i.e.
$$
\sum_{k\le 0}|\overline{\gamma}_{k+1}-\overline{\gamma}_{k}|+\sum_{k\ge 0}|\gamma_{k+1}-\gamma_{k}|<\infty,
$$ 
thus $((X_t)_{t\ge 0}, (P_x)_{x\in \mathbb{R}})$ is a semimartingale by Theorem \ref{semimart}(i), 
but for the $\alpha_k$ and $\overline{\alpha}_k$ corresponding to the  $l_k$, $k\ge 0$, and $r_k$, $k\le 0$, we have 
$$
\sum_{k\ge 0}|2\alpha_k -1|=\sum_{k\ge 0}\left |\frac{\gamma_{k+1}-\gamma_k}{\gamma_{k+1}+\gamma_k}\right |=\infty \ \ \mbox{ or }\ \ 
\sum_{k\le 0}|2\overline{\alpha}_k-1|=\sum_{k\le 0}\left |\frac{\overline{\gamma}_{k+1}-\overline{\gamma}_{k}}{\overline{\gamma}_{k+1}+\overline{\gamma}_{k}}\right |
=\infty.
$$
This happens typically if $\ell^0(X)\equiv 0$. Indeed, $\ell^0(X)\equiv 0$ implies 
the continuity of $a\mapsto \ell^a_t(X)$ in $a=0$ by \cite[VI.(1.7) Theorem]{ry05}. 
Thus $\lim_{k\to\infty} \ell_t^{l_k}(X)=0$ and $\lim_{k\to-\infty} \ell_t^{r_k}(X)=0$. 
This is for instance the case in Example \ref{Bessel} below with $\delta\in (1,2)$.\\ 
On the other hand $\sum_{k\ge 0}|2\alpha_k -1|+\sum_{k\le 0}|2\overline{\alpha}_k-1|<\infty$ is stronger than (\ref{semi1}) and (H1) together as it 
implies (\ref{semi1}) and (\ref{hbddvargamma}) below (cf. Remark \ref{leGallweaker3}(ii)) and then also (H1) holds (cf. Remark \ref{caph1}).
\end{rem}
\begin{exam}\label{skewbm}($\alpha$-skew Brownian motion)\\
Let  $\alpha\in (0,1)$, and  $\gamma_k = \frac{1-\alpha}{\alpha}$, $\overline{\gamma}_k=1$, for all $k\in \mathbb{Z}$, i.e. for $x\notin \{l_k, r_k;k\in \mathbb{Z}\}$
$$
\rho(x)=\frac{1-\alpha}{\alpha}1_{(-\infty,0)}(x)+1_{(0,\infty)}(x).
$$
Then, since the corresponding Dirichlet (form) norm is equivalent to the one of 
Brownian motion, we obtain that the corresponding process is conservative (even recurrent), and (H1) holds. Thus the corresponding process is a conservative diffusion by Proposition \ref{consdiff}. 
Moreover, clearly (\ref{semi1}) and (\ref{semi2}) hold with $\gamma=\frac{1-\alpha}{\alpha}$ and $\overline{\gamma}=1$. Hence by Theorem \ref{semimart} $((X_t)_{t\ge 0}, P_x)$ is a semimartingale 
for any $x\in \mathbb{R}$.
By Corollary  \ref{countablyskewbm} we have $N_t=(2\alpha-1)\ell_t^0(X)$, 
since $\alpha_k=\overline{\alpha}_k=\frac{1}{2}$ for all $k\in \mathbb{Z}$. Hence $((X_t)_{t\ge 0}, (P_x)_{x\in \mathbb{R}})$ is the $\alpha$-skew Brownian motion (cf e.g. \cite{hs}, \cite{itomckean}).
\end{exam}

\begin{exam}\label{Bessel}(Resemblance to Bessel processes)\\
In this example, we show that there is a  solution  to (\ref{weak3}) (which by definition is a conservative diffusion that is a semimartingale) with 
$$
\sum_{k\ge 0}|2\alpha_k-1|+\sum_{k\le 0}|2\overline{\alpha}_k-1|=\infty.
$$
Let $-l_k=r_{-k}=\frac{1}{k}$ for $k\ge 1$, $-l_{k}=r_{-k}=-k+2$ for $k\le 0$, and 
$\gamma_k=(-l_k)^{\delta-1}$, $\overline{\gamma}_k=(r_k)^{\delta-1}$, $k\in \mathbb{Z}$, $\delta\in (0,1)\cup (1,2)\cup [2,\infty)$. 
(The case $\delta=1$ corresponds to Brownian motion.) Then $\rho(x)$ is the upper Riemann step function of $\varphi(x):=|x|^{\delta-1}$ corresponding to the partition $(l_k)_{k\in \mathbb{Z}}$ on $(-\infty,0)$, and the 
lower Riemann step function of $\varphi$ corresponding to the partition $(r_k)_{k\in \mathbb{Z}}$ on $(0,\infty)$. We can hence easily see from \cite[Example 2.2.4]{fot} that 
$$
\mbox{cap}(\{0\})>0\Leftrightarrow \delta\in (0,2).
$$
By comparing the underlying Dirichlet form with the Dirichlet form of the Bessel processes (in this case $\rho(x)=|x|^{\delta-1}$), and using \cite[Theorem 4]{stu}, we can see that 
(H0) holds and so Proposition \ref{consdiff} applies. Moreover 
$$
\sum_{k\le -2}|\overline{\gamma}_{k+1}-\overline{\gamma}_{k}|=|1-\lim_{k\to \infty} k^{1-\delta}|   \ \ \mbox{ and }  \ \ \sum_{k\ge 1}|\gamma_{k+1}-\gamma_{k}|=|\lim_{k\to \infty}k^{1-\delta}-1|.
$$ 
Thus by Theorem \ref{semimart} the corresponding process is not a semimartingale if $\delta\in(0,1)$, and a semimartingale with respect to to all starting points that have positive capacity, if $\delta\ge 1$.\\
However (cf. Remark \ref{counterex}), if $\delta\in (1,2)$, then by the mean value theorem for some $\vartheta_k\in [k,k+1]$, $k\ge 1$, 
\begin{eqnarray*}
\sum_{k\le -2}\left |\frac{\overline{\gamma}_{k+1}-\overline{\gamma}_{k}}{\overline{\gamma}_{k+1}+\overline{\gamma}_{k}}\right | 
& = & \sum_{k\ge 1}\frac{ (k+1)^{\delta-1}-(k)^{\delta-1} }{(k+1)^{\delta-1}+(k)^{\delta-1} }\\
& = &\sum_{k\ge 1}\frac{ (\delta-1)\vartheta_k^{\delta-2} }{ (k+1)^{\delta-1}+(k)^{\delta-1} } \ge\sum_{k\ge 1}\frac{\delta-1}{2}(k+1)^{-1}=+\infty.\\
\end{eqnarray*}
Exactly in the same way we can show 
$$
\sum_{k\ge 0}\left |\frac{\gamma_{k+1}-\gamma_k}{\gamma_{k+1}+\gamma_k}\right |=\infty.
$$
Note that in this case $\ell_t^{0}(X)\equiv 0$, and $\ell^{a}_t(X)$, $a\in \{0,l_k,r_k, k\in \mathbb{Z}\}$ is uniquely associated to its Revuz measure $a^{\delta-1}\delta_{a}$. Moreover, $a_n^{\delta-1}\delta_{a_n}\to 0$ weakly whenever $a_n\to 0$.
\end{exam}
Our strategy to construct a solution to (\ref{weak3}) was first to construct a solution to the basic equation (\ref{weak1}) via the underlying Dirichlet form determined by (\ref{closureDF}), and then to rewrite (\ref{weak1}) as (\ref{weak3}) using (\ref{localtime(ii)}). Now, we ask under which assumptions on the underlying parameters a solution to (\ref{weak3}) exists.
\begin{prop}\label{legallweaker0}
Let $(\alpha_k)_{k\in\mathbb{Z}}, (\overline{\alpha}_k)_{k\in\mathbb{Z}}\subset (0,1)$,  
be arbitrarily given. Let $(l_k)_{k\in\mathbb{Z}}, (r_k)_{k\in\mathbb{Z}}$, be a partition of $\mathbb{R}$ as described at the beginning of Section \ref{two}. 
For arbitrarily chosen $\gamma_0>0$ and $\overline{\gamma}_0>0$ define 
\begin{equation}\label{alphagamma2}\tag{Gamdef0}
\gamma_k=\prod_{j=k}^{-1}\frac{1-\alpha_j}{\alpha_j}\gamma_0,\ \ k\le -1,\ \  \ \ \gamma_k=\prod_{j=0}^{k-1}\frac{\alpha_j}{1-\alpha_j}\gamma_0, \ \ k\ge 1.
\end{equation}
and  
\begin{equation}\label{alphagamma1}\tag{Gamdef1}
\overline{\gamma}_k=\prod_{j=k}^{-1}\frac{1-\overline{\alpha}_j}{\overline{\alpha}_j}\overline{\gamma}_0,\ \ k\le -1,\ \  \ \ \overline{\gamma}_k=\prod_{j=0}^{k-1}\frac{\overline{\alpha}_j}{1-\overline{\alpha}_j}\overline{\gamma}_0, \ \ k\ge 1,
\end{equation}
Suppose that (\ref{semi1}) holds for $(\gamma_k)_{k\ge 0}$, $(\overline{\gamma}_k)_{k\le 0}$ defined by (\ref{alphagamma2}), (\ref{alphagamma1}). 
Then the bilinear form in (\ref{closureDF}) with $\rho$ defined through 
$(\gamma_k)_{k\ge 0}$, $(\overline{\gamma}_k)_{k\le 0}$, $(l_k)_{k\in\mathbb{Z}}$, and $(r_k)_{k\in\mathbb{Z}}$ as above, is well defined and closable in $L^2(\mathbb{R};\rho dx)$. Suppose that its 
closure $({\cal E}, D({\cal E}))$ satisfies (H0) and (H1). Then there exists a conservative diffusion  $((X_t)_{t\ge 0}, (P_x)_{x\in \mathbb{R}})$, which is a semimartingale and which weakly solves (\ref{weak3}).
\end{prop}
\begin{proof}
Condition (\ref{semi1}) implies that $\rho$ defined through $(\gamma_k)_{k\ge 0}$, $(\overline{\gamma}_k)_{k\le 0}$, and $(l_k)_{k\in\mathbb{Z}}, (r_k)_{k\in\mathbb{Z}}$ as in the statement 
is in $L^1_{loc}(\mathbbm{R},dx)$. Therefore, exactly as explained after (\ref{closureDF}) the bilinear form (\ref{closureDF}) is well defined and closable in $L^2(\mathbb{R};\rho dx)$. Since the closure $({\cal E}, D({\cal E}))$ is regular by construction and moreover satisfies (H0) and (H1) by assumption, we can apply Corollary \ref{countablyskewbm}(i) to obtain the result.
\end{proof}
\begin{rem}\label{foranyalpha}
Suppose that all the conditions of Proposition \ref{legallweaker0} are satisfied. Let $\gamma, \overline{\gamma}$ be defined as in (\ref{semi2}) 
where $(\gamma_k)_{k\ge 0}$, $(\overline{\gamma}_k)_{k\le 0}$ is given by (\ref{alphagamma2}), (\ref{alphagamma1}). 
If $\gamma=0$, $\overline{\gamma}>0$ or $\gamma>0$, $\overline{\gamma}=0$, then $\alpha\in\{0,1\}$ in (\ref{n1}). If $\gamma=\overline{\gamma}=0$, then $\alpha=\frac{1}{2}$. 
If $\gamma,\overline{\gamma}>0$, then we can obtain a solution to (\ref{weak3}) for any $\alpha\in (0,1)$ by varying $\gamma_0, \overline{\gamma}_0$ in Proposition \ref{legallweaker0}. 
Note that the values of $\alpha_k$, $\overline{\alpha}_k$ are not influenced by varying $\gamma_0, \overline{\gamma}_0$, only $\alpha$ is influenced.
\end{rem}

\section{Pathwise uniqueness, ergodic properties and applications to advection-diffusion}\label{three}
In this section we investigate further properties of the process constructed in Section \ref{two} under more restrictive assumptions on the density $\rho$. 
It turns out that  (\ref{semi1}) and the below (\ref{hbddvargamma}) are the right framework under which this process is to be considered. Starting from these 
two assumptions as a basis, we derive sufficient conditions for pathwise uniqueness and sharp conditions for non-explosion, recurrence and positive recurrence. Having 
developed the necessary tools, we propose an application to advection-diffusions in layered media via the theory of generalized Dirichlet forms.\\ 
In order to fix the final assumptions that will be in force throughout this section (see right after Remark \ref{hbddvarrem} below), we first fix $\rho$ as in (\ref{gammadef}), such that 
(\ref{semi1}) holds. Note that then $\rho\in L^1_{loc}(\mathbb{R};dx)$ by (\ref{semi2}). Furthermore, we assume that 
for  $\gamma, \overline{\gamma}$ as defined in (\ref{semi2}) it holds $\gamma, \overline{\gamma}>0$.
The latter implies that $\frac{1}{\rho}\in L^1_{loc}(\mathbb{R},dx)$ and that (H1) holds (see Remark \ref{caph1}). In contrast to section \ref{two}, {\bf we do not assume (H0)}. In particular, according to Remark \ref{localtime}, we have that $((X_t)_{t\ge 0}, (P_x)_{x\in\mathbb{R}})$ is a semimartingale and a diffusion up to lifetime $\zeta=\inf\{t>0\,|\,X_t\notin \mathbb{R}\}$ and for all $x\in \mathbb{R}$ it holds that 
\begin{eqnarray}\label{weak1zeta}
X_t=x+W_t+ \sum_{k\in \mathbb{Z}}\left \{\frac{\gamma_{k+1}-\gamma_{k}}{2}\ell_t^{l_k}+\frac{\overline{\gamma}_{k+1}-\overline{\gamma}_{k}}{2}\ell_t^{r_k}\right \}+
\frac{\overline{\gamma}-\gamma}{2}\ell^0_t, \ t<\zeta,\ P_x\mbox{-a.s.}
\end{eqnarray}
\subsection{Conservativeness and pathwise uniqueness}\label{3.1}
Let $\alpha:=\frac{\overline{\gamma}}{\overline{\gamma}+\gamma}$. Suppose $h:\mathbb{R}\to \mathbb{R}$ is the difference of two convex functions and piecewise linear with slope 
$\frac{\alpha\gamma}{\gamma_{k+1}}$ on the interval $(l_k,l_{k+1})$ and slope 
$\frac{(1-\alpha)\overline{\gamma}}{\overline{\gamma}_{k+1}}$ on the interval $(r_k, r_{k+1})$, $k\in \mathbb{Z}$. In particular $h$ is continuous and uniquely determined up to a constant. In order to fix a version, we let 
$$
h(0)=0.
$$ 
Let $h'(x)=\frac{h'(x+)+h'(x-)}{2}$ denote the symmetric derivative of $h$. In particular 
$$
h'(0)=\frac{\lim_{k\to -\infty}\frac{(1-\alpha)\overline{\gamma}}{\overline{\gamma}_{k+1}}+\lim_{k\to \infty}\frac{\alpha\gamma}{\gamma_{k+1}}}{2}=\frac12.
$$  
\begin{rem}\label{hbddvarrem}
$h$ with the properties stated above exists, if and only if $h'$ is locally of bounded variation, that is
\begin{equation}\label{hbddvargamma}\tag{S1}
\sum_{k\ge 0} \left |\frac{1}{\gamma_{k}}-\frac{1}{\gamma_{k+1}}\right | +\sum_{k\le 0}\left |\frac{1}{\overline{\gamma}_{k}}-\frac{1}{\overline{\gamma}_{k+1}}\right |<\infty.
\end{equation}
Note further that all our assumptions so far (namely (\ref{semi1}), $\gamma,\overline{\gamma}>0$ and the existence of $h$ as above) are satisfied, if and only if (\ref{semi1}) and (\ref{hbddvargamma}) hold. 
\end{rem}
According to Remark \ref{hbddvarrem} we will {\bf assume} (to the sole exception of Theorem \ref{legallweaker1}) {\bf from now on up to the end of section \ref{three} that (\ref{semi1}) and (\ref{hbddvargamma}) hold and fix $h$ like above}.
\begin{lem}\label{algebraic0dir}
$(Y_t:=h(X_t))_{t\ge 0}$ is a continuous local martingale up to $\zeta$ with quadratic variation
\begin{eqnarray}\label{transformationdir}
\langle Y \rangle_t & = & \int_0^t (h'\circ h^{-1})^2(Y_s)ds,\ \ \ t<\zeta,\ \ P_x\mbox{-a.s.} \\ \nonumber
\end{eqnarray}
for all $x\in \mathbb{R}$.
\end{lem}
\begin{proof}
(Cf. proof of Theorem \ref{semimart}) Note that $h\in D({\cal E})_{loc}$, since it can be approximated locally in the Dirichlet space by its convolution with a standard Dirac sequence. For any $f\in C_0^{\infty}(\mathbb{R})$ we then calculate
\begin{eqnarray}\label{itotanakadir}
-{\cal E}(h,f) & = & -\frac{1}{2}\lim_{n\to\infty}\sum_{-n\le k\le n}\left \{\int_{l_k}^{l_{k+1}}\alpha \gamma f'(x)dx+\int^{r_{k+1}}_{r_{k}}(1-\alpha)\overline{\gamma}f'(x)dx\right \} \nonumber \\ 
& = & -\frac{1}{2}\lim_{n\to\infty}\left \{\alpha \gamma (f(l_{n+1})-f({l_{-n}})) +(1-\alpha)\overline{\gamma}(f(r_{n+1})-f(r_{-n}))\right \} 
\nonumber \\ 
& = & -\frac{1}{2}\left \{\alpha \gamma f(0)-(1-\alpha)\overline{\gamma}f(0)\right \}=0. \\ \nonumber
\end{eqnarray}	
Therefore the drift $N^{[h]}$ in the Fukushima decomposition of $h(X_t)-h(X_0)$ vanishes on account of \cite[Theorem 5.5.4]{fot}.
It then follows from \cite[Theorem 5.5.1]{fot} that $h(X_t)-h(X_0)=M^{[h]}_t$ is a continuous local martingale up to lifetime. Under (H1) the equality is strict, i.e. it holds $P_x$-a.s. for all $x\in\mathbb{R}$. 
The quadratic variation $\langle M^{[h]} \rangle$ of the local martingale $M^{[h]}$ can be identified by calculating its Revuz measure $\mu_{\langle M^{[h]} \rangle}$. We have
$$
\mu_{\langle M^{[h]} \rangle}=h'(x)^2\rho dx,
$$
which is the same than the Revuz measure of the additive functional $A_t=\int_0^t h'(X_s)^2 ds$. 
We hence obtain by the uniqueness of the Revuz correspondence and (H1) that 
$$
\langle M^{[h]} \rangle_t=\int_0^t h'(X_s)^2 ds, \ t<\zeta
$$
$P_x$-a.s. for all $x\in\mathbb{R}$. Writing $X_t=h^{-1}(Y_t)$ we obtain the final result.
\end{proof}\\ \\
Although, in our case we do not have a classical It\^o-equation, we shall call the function $h$ in Lemma \ref{algebraic0dir} {\it scale function} of the diffusion $((X_t)_{t\ge 0}, (P_x)_{x\in\mathbb{R}})$, and then the corresponding {\it speed measure} is 
$$
\mu(dy)=\frac{2}{h'(y)}dy.
$$
We let further for $x\in \mathbb{R}$
$$
\Phi(x):= \frac{1}{2}\int_0^x h'(z)\int_0^z \mu(dy)\,dz = \int_0^x \frac{h(x)-h(y)}{h'(y)}dy.
$$
Note that $\Phi$ is well defined and continuous, since $h$ is strictly increasing and continuous, and $h'$ is locally bounded and locally bounded away from zero by the assumption   
$\gamma, \overline{\gamma}>0$. Indeed the latter implies $h'(0)=\frac{1}{2}$. \\ \\
For a function $f:\mathbb{R}\to\mathbb{R} $ we define
$$
f(\infty):=\lim_{x\nearrow \infty}f(x)\ \ \mbox{and} \ \ \ f(-\infty):=\lim_{x\searrow -\infty}f(x)
$$
whenever the limits exist in $\mathbb{R}\cup\{\pm \infty\}$.
\begin{prop}\label{conservativenessdir}
The following are equivalent:
\begin{itemize}
 \item[(i)] $((X_t)_{t\ge 0}, (P_x)_{x\in\mathbb{R}})$ is conservative, i.e. $P_x(\zeta=\infty)=1\ \ \forall x\in\mathbb{R}$, where 
$\zeta=\inf\{t>0\,|\, X_t\notin (-\infty,\infty)\}=\inf\{t>0\,|\, Y_t\notin (h(-\infty),h(\infty))\}$
\item[(ii)] $\Phi(-\infty)=\Phi(\infty)=\infty$, i.e. $-\infty$ and $\infty$ are {\it non-exit (inaccessible) boundaries}.
\item[(iii)] There exist $u_n\in D(\cal{E})$, $n\ge 1$, $0\le u_n\nearrow 1$ $dx$-a.e. as $n\to \infty$ such that ${\cal{E}}(u_n,G_1 w)\to 0$ as $n\to \infty$ for some 
$w\in L^2(\mathbb{R},\rho dx)\cap  L^1(\mathbb{R},\rho dx)$ such that $w>0$ a.e. (Here $(G_{\alpha})_{\alpha>0}$ is the resolvent of $({\cal{E}}, D({\cal{E}}))$, see \cite{fot}).
\end{itemize}
\end{prop}
\begin{proof}
$(i)\Leftrightarrow (ii)$ is the well-known {\it Feller's test of non-explosions}. Although, in our case we do not have a classical It\^o-equation, 
it can be carried out exactly as in \cite[Section 6.2]{d}. Indeed, for its proof we mainly need the existence of a good scale function and speed measure, which is here the case.
Further, it is well-known in the theory of Dirichlet forms that $(iii)$ implies $P_x(\zeta=\infty)=1$ $\forall x\in\mathbb{R}\setminus N$, where $cap(N)=0$ (see \cite[Theorem 1.6.6 (iii)]{fot}). 
Under (H1) we must have $N=\emptyset$, hence $(iii)\Rightarrow (i)$. In order to see $(ii)\Rightarrow (iii)$, we can define $(u_n)_{n\ge 1}$ as follows. Let 
$a_n:=\int_{-n}^{0} \frac{h(y)-h(x)}{h'(y)}dy$, $b_n:=\int_{0}^{n}\frac{h(x)-h(y)}{h'(y)}dy$ and for $n\ge 1$

\[u_n(x):= \left\{ \begin{array}{r@{\quad}l}
 1-\frac{1}{a_n}\int_{x}^{0}  \frac{h(y)-h(x)}{h'(y)}dy&\mbox{if }\ x\in[-n,0], \\ 
 1-\frac{1}{b_n}\int_{0}^{x} \frac{h(x)-h(y)}{h'(y)}dy&\mbox{if }\ x\in [0,n],\\ 
 0&\mbox{elsewhere}. \end{array} \right. \] \\

Clearly $0\le u_n\nearrow 1$ $dx$-a.e. as $n\to \infty$. Fix a standard Dirac sequence $(\varphi_{\varepsilon})_{\varepsilon>0}$ and define 
$u_n^k:=\varphi_{\frac{1}{k}}*u_n$, $k\ge 1$. Then  $u_n^k\in C_0^{\infty}(\mathbb{R})$ and by standard properties of the convolution product 
one can easily see that $u_n^k \to u_n$ in $D({\cal E})$ as $k\to \infty$. Hence $u_n\in D({\cal E})$. For
$$
\lim_{n\to \infty}{\cal E}(u_n, G_1 w)= 0,
$$
see e.g. \cite[Lemma 3.1]{o92}.
\end{proof}
\begin{cor}\label{conservativeness2}
Property $(ii)$ of Proposition \ref{conservativenessdir} holds, if and only if
\begin{eqnarray*}
\lim_{n\to \infty}\sum_{l\le n}(r_{l+1}-r_{l})\left \{\frac{1}{2}(r_{l+1}-r_{l})+
\frac{1}{\overline{\gamma}_{l+1}}\sum_{k\le l-1}\overline{\gamma}_{k+1}(r_{k+1}-r_{k})\right \}=\infty
\end{eqnarray*}
and
\begin{eqnarray*}
\lim_{n\to \infty}\sum_{m\ge -n}(l_{m+1}-l_{m})\left \{ \frac{1}{2}(l_{m+1}-l_{m})+
\frac{1}{{\gamma}_{m+1}}\sum_{k\ge m+1} {\gamma}_{k+1}(l_{k+1}-l_{k})\right \}=\infty,
\end{eqnarray*}
i.e. in this case we have non-explosion for every starting point.
\end{cor}
\begin{proof}
We get for any $n\in \mathbb{Z}$
\begin{eqnarray*}
\Phi(r_{n+1}) & = & \int_0^{r_{n+1}} h'(z)\int_0^z \frac{1}{h'(y)}dy\,dz\\
& = & \sum_{l\le n}\int_{r_l}^{r_{l+1}} h'(z) \left \{\sum_{k \le l-1}\int_{r_k}^{r_{k+1}}\frac{1}{h'(y)}dy+\int_{r_l}^{z}\frac{1}{h'(y)}dy\right \}dz\\
& = & \sum_{l\le n}\left \{ \frac{1}{2}(r_{l+1}-r_{l})^2+\sum_{k\le l-1}\frac{\overline{\gamma}_{k+1}}{\overline{\gamma}_{l+1}}(r_{k+1}-r_{k})(r_{l+1}-r_{l})\right \}\\
\end{eqnarray*}
and similarly
\begin{eqnarray*}
\Phi(l_{n}) & = & \int_{l_n}^{0} h'(z)\int_z^0 \frac{1}{h'(y)}dy\,dz\\
& = & \sum_{m\ge n}\left \{ \frac{1}{2}(l_{m+1}-l_{m})^2+\sum_{k\ge m+1}\frac{{\gamma}_{k+1}}{{\gamma}_{m+1}}(l_{k+1}-l_{k})(l_{m+1}-l_{m})\right \}.\\
\end{eqnarray*}
Hence
$$
\Phi(\infty)=\lim_{n\to\infty}\Phi(r_{n+1})=\infty\ \ \mbox{ and } \ \ \Phi(-\infty)=\lim_{n\to -\infty}\Phi(l_{n})=\infty
$$
hold, if and only if the two conditions stated in the lemma are satisfied.
\end{proof}\\
For $l,m\in\mathbb{Z}$ let
$$
\overline{v}_l:=\left \{ \frac{1}{2}(r_{l+1}-r_{l})+\frac{1}{\overline{\gamma}_{l+1}}\sum_{k\le l-1}\overline{\gamma}_{k+1}(r_{k+1}-r_{k})\right \}
$$
and
$$
v_m:=\left \{ \frac{1}{2}(l_{m+1}-l_{m})+\frac{1}{{\gamma}_{m+1}}\sum_{k\ge m+1} {\gamma}_{k+1}(l_{k+1}-l_{k})\right \}.
$$
It follows immediately from Corollary \ref{conservativeness2} that a sufficient condition for conservativeness is given by: $\exists \delta>0$ such that
$$
\mbox{{\it either} } r_{l+1}-r_l\ge \delta  \mbox{ for infinitely many } l
\mbox{ {\it or} } \exists l_0\in\mathbbm{Z} \mbox{ with }  \overline{v}_l\ge \delta \mbox{ for all }l\ge l_0
$$
and
$$
\mbox{{\it either} } l_{k+1}-l_k\ge \delta  \mbox{ for infinitely many } k 
\mbox{ {\it or} } \exists m_0\in\mathbbm{Z} \mbox{ with } v_m\ge \delta \mbox{ for all }m\le m_0.
$$
For instance, if there exists $k_0,l_0\in \mathbb{Z}$ with $\inf_{l\ge l_0}(r_{l+1}-r_l)\ge \delta$ and  $\inf_{k\le k_0}(l_{k+1}-l_k)\ge \delta$, then conservativeness holds. 
However, under the conditions (\ref{semi1}) and (\ref{hbddvargamma}) conservativeness is suitably described as in the following remark.
\begin{rem}\label{conservativeglobalprop}
The conditions (\ref{semi1}) and (\ref{hbddvargamma}) are local conditions as they depend only on the local behavior of $\rho$ around the accumulation point zero. In particular (\ref{hbddvargamma}) is crucial for deriving pathwise uniqueness properties (see Theorem \ref{uniqueness} below). Note that the assumption (H0) (resp. (\ref{semi1})) in Theorem \ref{uniqueness} below can be seen as a formal condition  that are used to ensure uniqueness up to infinity (resp. the semimartingale property). Under the local conditions (\ref{semi1}) and (\ref{hbddvargamma}) the conditions in Corollary \ref{conservativeness2} are global ones and depend only on the behavior of $\rho$ outside arbitrarily large compact sets that contain the accumulation point. In fact, for any $n_0\in \mathbb{N}$, (\ref{semi1}) and (\ref{hbddvargamma}) 
imply that $(\gamma_k)_{k>-n_0}$ and $(\overline{\gamma}_{k})_{k<n_0}$  are bounded below and above by strictly positive constants and moreover 
$\sum_{l<n_0}(r_{l+1}-r_{l})=r_{n_0}$,  $\sum_{k>-n_0}(l_{k+1}-l_{k})=l_{-n_0+1}$. From this it is then not difficult to see that the conditions of Corollary \ref{conservativeness2} are equivalent to the following ones: 
\begin{equation}\label{global1}\tag{C0}
\lim_{n\to \infty}\sum_{l=n_0}^{n}\frac{r_{l+1}-r_{l}}{\overline{\gamma}_{l+1}}\left (\sum_{k=n_0}^{l}\overline{\gamma}_{k+1}(r_{k+1}-r_{k})\right )=\infty
\end{equation}
and
\begin{equation}\label{global2}\tag{C1}
\lim_{n\to \infty}\sum_{m=-n}^{-n_0}\frac{l_{m+1}-l_{m}}{{\gamma}_{m+1}}\left ( \sum_{k= m}^{-n_0} {\gamma}_{k+1}(l_{k+1}-l_{k})\right )=\infty,
\end{equation}
for one and hence any $n_0\in \mathbb{N}$, where as usually $\sum_{k=l}^{m}:=0$ for $m<l$. \\
\end{rem}
\begin{exam}\label{counterexample} 
Let us give an example where we have explosion. Let $r_l:=\sum_{k=1}^l \frac{1}{k}$, $l\ge 1$ and $\overline{\gamma}_{k+1}=C^{k}(k+1)$, $k\ge 1$,  where $C>1$ is some constant, and let the remaining $r_l,\overline{\gamma}_{k+1}, l_k, \gamma_{k+1}$ be just chosen such that conditions (\ref{semi1}) and (\ref{hbddvargamma}) are satisfied. Then 
\begin{eqnarray*}
\sum_{l=1}^{\infty}\frac{r_{l+1}-r_{l}}{\overline{\gamma}_{l+1}}\left (\sum_{k=1}^{l}\overline{\gamma}_{k+1}(r_{k+1}-r_{k})\right ) & = & 
\sum_{l=1}^{\infty}\frac{1}{C^l (l+1)^2}\left (\frac{C^{l+1}-C}{C-1}\right )<\infty,\\
\end{eqnarray*}
and so according to (\ref{global1}) in Remark \ref{conservativeglobalprop} with $n_0=1$ it follows that we cannot have conservativeness.
Note that in this example, $(r_l)_{l\in\mathbb{Z}}$ has an accumulation point at \lq\lq infinity\rq\rq\ and the skew reflection is 
with $\overline{\alpha}_{k}\approx \frac{1}{1+\frac{1}{C}}>\frac{1}{2}+\varepsilon$ for  $k\ge N$ for some $N\in\mathbb{N}$. 
\end{exam}
\begin{lem}\label{algebraic0}
Suppose that additionally to (\ref{semi1}) and (\ref{hbddvargamma}), $({\cal E},D({\cal E}))$ is conservative. Then 
$(Y_t:=h(X_t))_{t\ge 0}$ satisfies $P_x$-a.s
\begin{eqnarray}\label{transformation}
Y_t & = & h(x)+\int_0^t \sum_{k\in \mathbb{Z}}\left (\frac{\alpha\gamma}{\gamma_{k+1}} 1_{[l_k,l_{k+1})}+\frac{(1-\alpha)\overline{\gamma}}{\overline{\gamma}_{k+1}} 1_{[r_k,r_{k+1})}\right )\circ h^{-1}(Y_s)dW_s \\ \nonumber
\end{eqnarray}
for all $t\ge 0$ and all $x\in \mathbb{R}$. 
\end{lem}
\begin{proof}
Let $h''(da)$ the signed measure that is induced by the second derivative of $h$.
Then applying the symmetric version of \cite[VI. (1.5) Theorem]{ry05} with $h$ and (\ref{weak3}) we obtain $P_x$-a.s.
\begin{eqnarray}\label{itotanaka}
h(X_t) & = & h(x)+\int_0^t h'(X_s)dX_s +\frac{1}{2}\int_{\mathbb{R}} \ell_t^a(X)h''(da) \nonumber\\
& = & h(x)+ \int_0^t h'(X_s)dW_s+(2\alpha-1)\int_{0}^t h'(X_s)d\ell_s^0(X)\nonumber \\
&+& \sum_{k\in \mathbb{Z}}\left \{(2\alpha_k -1)\int_{0}^t h'(X_s)d\ell_s^{l_k}(X)+(2\overline{\alpha}_k-1)\int_{0}^t h'(X_s)d\ell_s^{r_k}(X)\right \}
\nonumber \\
&+& \sum_{k\in \mathbb{Z}}\left \{\frac{\frac{\alpha\gamma}{\gamma_{k+1}}-\frac{\alpha\gamma}{\gamma_{k}}}{2}\ell_t^{l_k}(X)+\frac{\frac{(1-\alpha)\overline{\gamma}}{\overline{\gamma}_{k+1}}-\frac{(1-\alpha)\overline{\gamma}}{\overline{\gamma}_{k}}}{2}\ell_t^{r_k}(X)\right \}
+ \frac{(1-\alpha)-\alpha}{2}\ell_t^{0}(X)
\nonumber \\
& = & h(x)+ \int_0^t h'(X_s)dW_s+\left \{(2\alpha-1)\frac{1}{2}+\frac{1-2\alpha}{2}\right \}\ell_t^0(X)\nonumber \\
& & + \sum_{k\in \mathbb{Z}}\left ((2\alpha_k -1)\frac{\frac{\alpha\gamma}{\gamma_{k+1}}+\frac{\alpha\gamma}{\gamma_{k}}}{2}+\frac{\frac{\alpha\gamma}{\gamma_{k+1}}-\frac{\alpha\gamma}{\gamma_{k}}}{2}\right )\ell_s^{l_k}(X)
\nonumber \\
& &\ + \sum_{k\in \mathbb{Z}}\left ((2\overline{\alpha}_k-1)\frac{\frac{(1-\alpha)\overline{\gamma}}{\overline{\gamma}_{k+1}}+\frac{(1-\alpha)\overline{\gamma}}{\overline{\gamma}_{k}}}{2}+\frac{\frac{(1-\alpha)\overline{\gamma}}{\overline{\gamma}_{k+1}}-\frac{(1-\alpha)\overline{\gamma}}{\overline{\gamma}_{k}}}{2}\right )\ell_t^{r_k}(X) \nonumber \\ 
& = & h(x)+ \int_0^t h'(X_s)dW_s.\\ \nonumber
\end{eqnarray}
Now the statement follows from (\ref{n3}).
\end{proof}
\begin{theo}[Starting from the Dirichlet form]\label{uniqueness}
Suppose that additionally to (\ref{semi1}) and (\ref{hbddvargamma}), $({\cal E},D({\cal E}))$ is conservative.
Then strong uniqueness holds for (\ref{weak1}) and (\ref{weak3}), i.e pathwise uniqueness holds for (\ref{weak1}) and (\ref{weak3}) and 
there exists a unique strong solution to (\ref{weak1}) and (\ref{weak3}).
\end{theo}
\begin{proof}
Let  $\widetilde{\sigma}:=\sum_{k\in \mathbb{Z}}\left (\frac{\alpha\gamma}{\gamma_{k+1}} 1_{[l_k,l_{k+1})}+\frac{(1-\alpha)\overline{\gamma}}{\overline{\gamma}_{k+1}} 1_{[r_k,r_{k+1})}\right )\circ h^{-1}$, $\sigma:=\widetilde{\sigma}\circ h$, and $h$ like in Lemma \ref{algebraic0}. 
By \cite[Remarques: b), p. 21]{lg83} (see also \cite[IX.(3.5) Theorem iii) and (3.13) Exercise]{ry05}, or even \cite{nakao} that we could use with a localization procedure), we know that pathwise uniqueness holds for (\ref{transformation}), if $\widetilde{\sigma}$ is locally bounded away from zero and locally of finite 
quadratic variation. Of course, it is enough to check this in a neighborhood of zero and for $\sigma$ instead of $\widetilde{\sigma}$, since $h^{-1}$ is strictly increasing and continuous in a neighborhood of zero and $h^{-1}(0)=0$.
Since 
$$
\lim_{k\to -\infty}\frac{(1-\alpha)\overline{\gamma}}{\overline{\gamma}_{k+1}}=1-\alpha,\ \ \ \lim_{k\to \infty}\frac{\alpha\gamma}{\gamma_{k+1}}=\alpha,
$$ 
and $\alpha\in (0,1)$, we clearly have that $\sigma$ is locally bounded away from zero in any neighborhood of zero. If (\ref{hbddvargamma}) is satisfied, then $\sigma$ is locally of finite variation around zero, hence in particular locally of finite quadratic variation around zero. Thus the result follows by \cite[Remarques: b), p. 21]{lg83}. 
Since $h$ is a continuous bijection on its image with $h(0)=0$, and 
$X:=h^{-1}(Y)$  with $Y$ like in (\ref{transformation}) solves (\ref{weak3}), pathwise uniqueness also holds for (\ref{weak3}). By the Yamada-Watanabe Theorem there exists a unique strong solution 
$Y$ to (\ref{transformation}), hence strong existence and pathwise uniqueness also holds for $X:=h^{-1}(Y)$. Since (\ref{weak3}) is just (\ref{weak1}) rewritten with the symmetric local times, strong uniqueness also holds for (\ref{weak1}).
\end{proof}\\ \\
In the following theorem, we do not assume from the beginning
(\ref{semi1}) and (\ref{hbddvargamma}), which were in force throughout the subsection. We also do not assume that $({\cal E},D({\cal E}))$ is conservative from the beginning.
\begin{theo}{\bf (Starting from the SDE)}\label{legallweaker1} 
Let $(l_k)_{k\in\mathbb{Z}}, (r_k)_{k\in\mathbb{Z}}$ be a partition of $\mathbb{R}$ as described at the beginning of section \ref{two}.
Let $(\alpha_k)_{k\in\mathbb{Z}}, (\overline{\alpha}_k)_{k\in\mathbb{Z}}\subset (0,1)$. Suppose
\begin{equation}\label{legallcond1}\tag{LGloc}
\sum_{k\ge 0}|2\alpha_k-1|+\sum_{k\le 0}|2\overline{\alpha}_k-1|<\infty
\end{equation}
and that (\ref{global1}), (\ref{global2}) are satisfied for $(\gamma_k)_{k\in \mathbbm{Z}}$,
$(\overline{\gamma}_k)_{k\in \mathbbm{Z}}$ given by (\ref{alphagamma2}), (\ref{alphagamma1}). Then for any $\alpha\in (0,1)$ there exists 
a unique strong solution to (\ref{weak3}).
\end{theo}
\begin{proof}
Assume we can show (\ref{semi1}) and (\ref{hbddvargamma}) for $(\gamma_k)_{k\in \mathbbm{Z}}$,
$(\overline{\gamma}_k)_{k\in \mathbbm{Z}}$ given by (\ref{alphagamma2}), (\ref{alphagamma1}). Then $\exists\lim_{k\to \infty}\gamma_k, \exists\lim_{k\to -\infty}\overline{\gamma}_k>0$ and so (H1) holds by Remark \ref{caph1} for the regular Dirichlet form $({\cal E}, D({\cal E}))$ corresponding to $\rho$ in (\ref{gammadef}) with the above data. Conditions (\ref{global1}), (\ref{global2}) are equivalent to (H0) according to Remark \ref{conservativeglobalprop}. Hence we obtain existence of a solution to (\ref{weak3}) for any $\alpha\in (0,1)$ by Proposition \ref{legallweaker0} and Remark \ref{foranyalpha}. Strong uniqueness then follows from Theorem \ref{uniqueness}.\\
Now, we show that (\ref{semi1}) and (\ref{hbddvargamma}) hold. By symmetry it is enough to show that $\sum_{k\ge 0}|2\alpha_k-1|<\infty$ implies $\sum_{k\ge 0}\left (|\gamma_{k+1}-\gamma_k|+\left |\frac{1}{\gamma_{k+1}}-\frac{1}{\gamma_k}\right |\right )<\infty$. We have 
$$
\gamma_k=\gamma_0\prod_{j=0}^{k-1}\left ( 1+\beta_j\right ), \ \mbox{ where }\ \beta_j:= \frac{2\alpha_j-1}{1-\alpha_j}, \ j\ge 0.
$$
Since $\sum_{k\ge 0}|2\alpha_k-1|<\infty$ it follows easily $\sum_{k\ge 0}|\beta_k|<\infty$. Let $N\in \mathbbm{N}$ be such that $|\beta_k|<1$ for all $k\ge N$. For $|z|<1$ we have
\begin{eqnarray*}
\log(1+z)=z+z^2\underbrace{\left (-\frac12+\frac{z}{3}-\frac{z^2}{4}+...\right )}_{:=f(z)}
\end{eqnarray*}
and $f$ is continuous at $0$ with $\lim_{z\to 0}f(z)=-\frac12$. Thus $(f(\beta_k))_{k\ge N}$ converges to $-\frac12$ and is therefore bounded. It follows that $\sum_{k\ge N}\beta_k^2 f(\beta_k)$ converges absolutely.
Since $\log(1+\beta_j)=\beta_j+\beta_j^2 f(\beta_j)$ for $j\ge N$, we have that $\sum_{j\ge N}\log(1+\beta_j)$ converges absolutely. In particular $(\gamma_k)_{k\ge 0}$ converges. But then 
$$
\sum_{k\ge 0}|\gamma_{k+1}-\gamma_k|=\sum_{k\ge 0}|2\alpha_k-1|(\gamma_{k+1}+\gamma_k)<\infty.
$$
Since
$$
\frac{1}{\gamma_k}=\frac{1}{\gamma_0}\prod_{j=0}^{k-1}\frac{1-\alpha_j}{\alpha_j}=
\frac{1}{\gamma_0}\prod_{j=0}^{k-1}\left (1+\frac{1-2\alpha_j}{\alpha_j}\right ),
$$
we obtain similarly that $\left (\frac{1}{\gamma_k}\right )_{k\ge 0}$ converges and then
$$
\sum_{k\ge 0}\left |\frac{1}{\gamma_{k+1}}-\frac{1}{\gamma_k}\right |
=\frac{1}{\gamma_0}\sum_{k\ge 0}\left |\prod_{j=0}^{k-1}\left (\frac{1-\alpha_j}{\alpha_j}\right )\left (\frac{1-\alpha_k}{\alpha_k}-1\right )\right |
=\frac{1}{\gamma_0}\sum_{k\ge 0}\frac{1}{\gamma_k}\left |\frac{1-2\alpha_k}{\alpha_k}\right |<\infty.
$$
\end{proof}\\ \\
In \cite{lg84} Le Gall considered equations of type
\begin{eqnarray}\label{legall0}
X_t=x + \int_0^t \sigma(X_t)dW_t + \int_{\mathbb{R}} \ell_t^a(X)\mu(da)
\end{eqnarray}
where $\sigma$ is of bounded variation, bounded away from zero and right continuous, and $\mu$ is a signed measure of bounded total variation such that
$|\mu(\{a\})|<1$ for any $a\in \mathbb{R}$. Under these global assumptions (that imply in particular conservativeness) Le Gall showed weak existence and pathwise uniqueness for (\ref{legall0}). Hence by Le Gall's results we know that weak existence and pathwise uniqueness holds for (\ref{dist}), if (for $(\alpha_k)_{k\in \mathbb{Z}}$ as in (\ref{dist}))
\begin{eqnarray}\label{legall1}
\sum_{k\in \mathbb{Z}}|2\alpha_k-1|<\infty.
\end{eqnarray}
Le Gall's results do not cover in whole generality equation (\ref{dist}), since in equation (\ref{dist}) no assumption on the finiteness of $\sum_{k\in \mathbb{Z}}|2\alpha_k-1|$ is made. On the other hand, the results in \cite{lg84} allow for an accumulation point of the sequence $(z_k)_{k\in \mathbb{Z}}$ in (\ref{dist}), because (\ref{legall1}) is sufficient for weak existence and pathwise uniqueness of (\ref{legall0}) with $\sigma\equiv 1$. 
But (\ref{legall1}) is qualitatively stronger than our assumptions in Theorems \ref{uniqueness} and \ref{legallweaker1} as we explain in the following remark. 
\begin{rem}\label{leGallweaker3}
(i) If (\ref{legallcond1}) holds globally, i.e. if
\begin{equation}\label{legallcond2}\tag{LG}
\sum_{k\in \mathbb{Z}}(|2\alpha_k-1|+|2\overline{\alpha}_k-1|)<\infty,
\end{equation}
then (\ref{global1}), (\ref{global2}) automatically hold, because $(\gamma_k)_{k\in \mathbbm{Z}}$,
$(\overline{\gamma}_k)_{k\in \mathbbm{Z}}$ are bounded below and above by strictly positive constants. 
Indeed, this can be shown exactly as in the proof of Theorem \ref{legallweaker1}. 
Hence we recover qualitatively Le Gall's strong uniqueness results according to (\ref{legall1}) by Theorem \ref{legallweaker1}.
Here, we used the word \lq\lq qualitatively\rq\rq\ because of the following. Condition (\ref{legall1}) covers also the case of multiple accumulation points, as long as only the sum in (\ref{legall1}) remains finite. However, we could have considered this situation even with no finiteness condition on the sums in a straightforward manner. But since apart from notational complication no new phenomena will occur locally by considering even countably many isolated accumulation points, 
we excluded the case of multiple accumulation points for the convenience of the reader. \\
(ii) It can be seen from the proof of Theorem \ref{legallweaker1} that (\ref{semi1}) together with (\ref{hbddvargamma}) are equivalent to (\ref{legallcond1}) and then under either one of these equivalent assumptions, (H0) is equivalent to
(\ref{global1}), (\ref{global2}) for $(\gamma_k)_{k\in \mathbbm{Z}}$,
$(\overline{\gamma}_k)_{k\in \mathbbm{Z}}$ given by (\ref{alphagamma2}), (\ref{alphagamma1})
(cf. Proposition \ref{conservativenessdir}, Remark \ref{conservativeglobalprop}). 
Therefore, the assumptions of Theorem \ref{uniqueness} and Theorem \ref{legallweaker1} are equivalent. But Le Gall's global condition  (\ref{legallcond2}) is stronger than our assumptions in Theorem \ref{legallweaker1}. One can say that the assumptions in Theorem \ref{legallweaker1} consist of two types of assumptions. A local assumption (\ref{legallcond1}), to ensure pathwise uniqueness, and a global assumption (\ref{global1}) together with (\ref{global2}) to ensure non-explosion of the solution. Indeed, our strategy is similar to the one used in \cite{lg84}. With the help of a nice function, we transform our equation into a local martingale (see (\ref{transformation})) and then 
obtain uniqueness (cf. proof of Theorem \ref{uniqueness}). Since our assumptions are only local, we need some global control, i.e. non-explosion. This is our additional contribution to the work of Le Gall in \cite{lg84}.\\
\end{rem}

\subsection{Recurrence and transience}
In this subsection, we assume throughout that (\ref{semi1}) and (\ref{hbddvargamma}) hold. We define 
$$
D_y :=\inf \{t\ge0\,|\,X_t=y\},\ \ \ y\in \mathbb{R}.
$$
Under the assumptions (\ref{semi1}) and (\ref{hbddvargamma}), the scale function $h$ always exists. Therefore, exactly as in \cite[Chapter 6, Lemma (3.1)]{d} we can show that
\begin{eqnarray}\label{rec0}
P_x(D_a \wedge D_b<\infty)=1, \ \ \ \forall x\in (a,b).
\end{eqnarray}
It follows in particular that $((h(X_{t\wedge D_a \wedge D_b}))_{t\ge 0}, (P_x)_{x\in\mathbb{R}})$, with $h$ like in Lemma \ref{algebraic0}, is a uniformly bounded local 
martingale and by standard calculations it is well-known that for any $x\in (a,b)$
\begin{eqnarray}\label{rec1}
P_x(D_a < D_b)& = & \frac{h(b)-h(x)}{h(b)-h(a)} \\ \nonumber
\end{eqnarray}
and
\begin{eqnarray}\label{rec2}
P_x(D_b < D_a)& = & \frac{h(x)-h(a)}{h(b)-h(a)}. \\ \nonumber
\end{eqnarray}
\begin{theo}\label{recurrence}
The following are equivalent:
\begin{itemize}
 \item[(i)] $((X_t)_{t\ge 0}, (P_x)_{x\in\mathbb{R}})$ is recurrent, i.e. $P_x(D_y<\infty)=1\ \ \forall x,y\in\mathbb{R}$.
\item[(ii)] $h(-\infty)=-\infty$ and $h(\infty)=\infty$.
\item[(iii)] $\sum_{k\in\mathbb{Z}}\frac{l_{k+1}-l_k}{\gamma_{k+1}}=\infty$ and $\sum_{k\in\mathbb{Z}}\frac{r_{k+1}-r_k}{\overline{\gamma}_{k+1}}=\infty$.
\item[(iv)] $\int_{-\infty}^0 \frac{1}{\rho(x)}dx=\infty$ and  $\int_0^{\infty} \frac{1}{\rho(x)}dx=\infty$.
\item[(v)] There exist $u_n\in D(\cal{E})$, $n\ge 1$, $0\le u_n\nearrow 1$ $dx$-a.e. as $n\to \infty$ such that ${\cal{E}}(u_n,u_n)\to 0$ as $n\to \infty$.
\end{itemize}
\end{theo}
\begin{proof}
(H1) implies that $P_x$-a.s. $D_a \to \infty$ as $a \to + \infty$ or $a\to -\infty$. Hence by (\ref{rec1}) $h(\infty)=\infty$ is equivalent to  $P_x(D_a<\infty)=1$ for any $x\in (a,b)$, and 
by (\ref{rec2}) $h(-\infty)=-\infty$ is equivalent to  $P_x(D_b<\infty)=1$ for any $x\in (a,b)$. This is clearly equivalent to the recurrence of $((X_t)_{t\ge 0}, (P_x)_{x\in\mathbb{R}})$, 
hence $(i)\Leftrightarrow (ii)$. $(ii)\Leftrightarrow (iii)\Leftrightarrow (iv)$ is obvious. $(iv)\Rightarrow (v)$ is a special case of \cite[Theorem 2.2 (i)]{GiTr13}. For the reader's convenience, we include the proof. Let $a_n:=\int_{-n}^{0} \frac{1}{\rho(x)}dx$, $b_n:=\int_{0}^{n} \frac{1}{\rho(x)}dx$ and for $n\ge 1$

\[u_n(x):= \left\{ \begin{array}{r@{\quad}l}
 1-\frac{1}{a_n}\int_{x}^{0} \frac{1}{\rho(x)}dx&\mbox{if }\ x\in[-n,0], \\ 
 1-\frac{1}{b_n}\int_{0}^{x} \frac{1}{\rho(x)}dx&\mbox{if }\ x\in [0,n],\\ 
 0&\mbox{elsewhere}. \end{array} \right. \] \\
Clearly $0\le u_n\nearrow 1$ $dx$-a.e. as $n\to \infty$. Fix a standard Dirac sequence $(\varphi_{\varepsilon})_{\varepsilon>0}$ and define 
$u_n^k:=\varphi_{\frac{1}{k}}*u_n$, $k\ge 1$. Then  $u_n^k\in C_0^{\infty}(\mathbb{R})$ and by standard properties of the convolution product 
one can easily see that $u_n^k \to u_n$ in $D({\cal E})$ as $k\to \infty$. Hence $u_n\in D({\cal E})$. Furthermore
$$
{\cal E}(u_n,u_n)=\frac{1}{2}\int_{-n}^0\frac{1}{a_n^2}\frac{1}{\rho(x)}dx+\frac{1}{2}\int_{0}^n\frac{1}{b_n^2}\frac{1}{\rho(x)}dx=
\frac{1}{2}\left (\frac{1}{a_n}+\frac{1}{b_n}\right )\to 0
$$
as $n\to \infty$. $(v)\Rightarrow (i)$ is well known (see e.g. \cite{fot}).
\end{proof}
\begin{lem}\label{strongergodicity}
Let one of the conditions of Theorem \ref{recurrence} be satisfied. Let $(\theta_t)_{t\ge 0}$ be the shift operator of $((X_t)_{t\ge 0}, (P_x)_{x\in\mathbb{R}})$. 
Then for any $x,y\in \mathbb{R}$
\begin{eqnarray}
\lim_{t\to\infty}\sup_{A\in \cal{F}}|P_x\circ \theta_t^{-1}(A)-P_y\circ \theta_t^{-1}(A)|=0. 
\end{eqnarray}
\end{lem}
\begin{proof}
By Theorem \ref{recurrence}(i) $((X_t)_{t\ge 0}, {\cal F}, ({\cal F}_t)_{t\ge 0}, \zeta, (P_x)_{x\in \mathbb{R}})$ is a regular, recurrent diffusion in the sense of \cite{ka02}. Therefore the statement follows from
\cite[Lemma 23.17]{ka02}.
\end{proof}
\begin{rem}
The Dirichlet form $({\cal E}, D({\cal E}))$ is irreducible (see \cite{fot} for the definition). 
Therefore, by \cite[Lemma 1.6.4.(iii)]{fot}, it is either recurrent or transient. Thus 
Theorem \ref{recurrence} provides also sharp conditions about transience in the sense of \cite{fot}. 
\end{rem}
Let $(p_t(x,dy))_{t\ge 0}$ be the transition kernels corresponding to $((X_t)_{t\ge 0}, (P_x)_{x\in\mathbb{R}})$. Let $A\subset \mathbb{R}$ be Borel measurable such that $\int_A \rho(x)dx<\infty$. 
Since {\bf (H1)} holds, (we may assume that) $p_t 1_A(x):=p_t(x,A)\in D(\cal{E})$ is continuous in $x$ for any $t>0$. (If not we could choose continuous versions and construct 
a process via Kolmogorov's method. This process would then be indistinguishable form  $((X_t)_{t\ge 0}, (P_x)_{x\in\mathbb{R}})$). 
In particular the transition kernels have a density with respect to reference measure 
$m(dx):=\rho(x)dx$, since $m(A)=0$ implies $p_t(x,A)=0$ for $m$-a.e. $x$, hence every $x$ by continuity and full support of $m$. \\
Let $\cal{B}(\mathbb{R})$ be the Borel $\sigma$-algebra of $\mathbb{R}$. For a positive measure $\mu$ on $(\mathbb{R},\cal{B}(\mathbb{R}))$ and $t>0$, we define
$$
\mu p_t(A):=\int_{\mathbb{R}} p_t(x,A) \mu(dx),\ \ \ A\in \cal{B}(\mathbb{R}).
$$
$\mu$ is called an {\it invariant measure}, if $\mu p_t =\mu$ for any $t>0$. It is called an {\it invariant distribution}, if additionally $\mu(\mathbb{R})=1$. Clearly, the reference measure 
$m=\rho dx$ is invariant since by symmetry of $(p_t)_{t\ge 0}$ with respect to $m$ and conservativeness
$$
mp_t(A)=\int_{\mathbb{R}} p_t 1_A(x) m(dx)=\int_{\mathbb{R}} 1_A(x) P_x(X_t\in \mathbb{R})m(dx)=m(A),\ \ \ t>0.
$$
Suppose $((X_t)_{t\ge 0}, (P_x)_{x\in\mathbb{R}})$ is recurrent. Then $((X_t)_{t\ge 0}, (P_x)_{x\in\mathbb{R}})$ 
is called {\it null-recurrent} if 
$$
\lim_{t\to \infty}p_t 1_K (x)=\lim_{t\to \infty}P_x(X_t \in K)=0
$$ 
for any $x\in\mathbb{R}$ and any compact set 
$K$ with non-empty interior. Otherwise it is called {\it positive recurrent}.\\
It follows from the proof of Theorem \ref{recurrence2} $(iv)\Rightarrow (i)$ below, that if $((X_t)_{t\ge 0}, (P_x)_{x\in\mathbb{R}})$ is recurrent, then it is positive recurrent, if and only if 
$P_x(X_t \in K)$ does not converge to zero as $t\to\infty$ for any $x\in\mathbb{R}$ and any compact set 
$K$ with non-empty interior.
\begin{theo}\label{recurrence2}
Suppose $((X_t)_{t\ge 0}, (P_x)_{x\in\mathbb{R}})$ is recurrent. 
Then the following are equivalent:
\begin{itemize}
 \item[(i)] $((X_t)_{t\ge 0}, (P_x)_{x\in\mathbb{R}})$ is positive recurrent.
\item[(ii)] $\int_{-\infty}^{\infty}\frac{1}{h'(x)}dx<\infty$.
\item[(iii)] The invariant measure $\rho dx$ is finite, i.e. $\sum_{k\in\mathbb{Z}}\{\gamma_{k+1}(l_{k+1}-l_k)+\overline{\gamma}_{k+1}(r_{k+1}-r_k)\}<\infty$.
\item[(iv)] $p_t(x,dy)=P_x(X_t\in \cdot)$ converges weakly to the invariant distribution $\frac{\rho dx}{\int_{\mathbb{R}}\rho(x)dx}$ as $t\to \infty$ for any $x\in \mathbb{R}$.
\item[(v)] $E_x[D_y]<\infty$ $\forall x,y\in\mathbb{R}$.
\end{itemize}
\end{theo}
\begin{proof}
$(ii)\Leftrightarrow (iii)$ is obvious. $(ii)\Leftrightarrow (iv)$ follows from \cite[IV.4. Theorem 7]{man66}. (In order to facilitate comparison we note that the $m$ of \cite{man66} 
writes as $m(s)= \int_{0}^{s}\frac{2}{h'(x)}dx$, and that the $p$ of \cite{man66} is just our $h$). 
$(iv)\Rightarrow (i)$ follows easily from the Portemanteau-Theorem and we may use \cite[Chapter 5.5. D, Exercise 5.40 (i)]{ks91} or \cite[IV.4 (55), IV.3 (46)]{man66}
to obtain $(ii)\Leftrightarrow (v)$.\\
If $(i)$ is satisfied then we can find $t_n\nearrow \infty$ as $n\to \infty$, $x\in \mathbb{R}$, and a compact set 
$K_0$ such that $\inf_{n\ge 1}p_{t_n}(x,K_0)>0$. By Helly's Theorem we can find another subsequence, again denoted by $(t_n)_{n\ge 1}$ and a 
subprobability measure $\mu$, such that
$$
p_{t_n}(x,\cdot)\longrightarrow \mu \ \ \mbox{ weakly as }\ n\to \infty.
$$
The weak convergence holds indeed for any $x\in \mathbb{R}$ by Lemma \ref{strongergodicity}.
Thus for any open set $U$ and any compact set $K$, we have by the Portemanteau-Theorem that 
$\liminf_{n\to\infty}p_{t_n}1_{U}(x)\ge \mu(U)$, and $\limsup_{n\to\infty}p_{t_n}1_{K}(x)\le \mu(K)$ for any $x\in \mathbb{R}$. 
In particular, $\mu(U_0)\ge \mu(K_0)>0$ for any relatively compact (open) set $U_0$ containing $K_0$. 
Then, by Fatou's lemma, conservativeness, and symmetry of $(p_t)_{t\ge 0}$ with respect to $\rho dx$
\begin{eqnarray}\label{U}
\int_{\mathbb{R}} 1_U \rho(x)dx & = & \liminf_{n\to \infty}\int_{\mathbb{R}} p_{t_n}1_U (x)\rho(x)dx\nonumber \\
& \ge & \int_{\mathbb{R}} \liminf_{n\to \infty}p_{t_n}1_U (x)\rho(x)dx\nonumber\\
& \ge & \mu(U)\int_{\mathbb{R}}\rho(x)dx.
\end{eqnarray}
Applying (\ref{U}) with $U=U_0$ we conclude that $\int_{\mathbb{R}}\rho(x)dx<\infty$ and then $\mu(U)\le\frac{\int_{U}\rho(x)dx}{\int_{\mathbb{R}}\rho(x)dx}$ for any open set $U$. Similarly to (\ref{U}) we derive
\begin{eqnarray}\label{K}
\mu(K)\ge\frac{\int_{K}\rho(x)dx}{\int_{\mathbb{R}}\rho(x)dx}
\end{eqnarray}
for any compact set $K$. Hence by inner regularity of the measures it follows $\mu(B)\ge\frac{\int_{B}\rho(x)dx}{\int_{\mathbb{R}}\rho(x)dx}$ 
for any Borel set $B$, which further implies that $\mu=\rho dx$. Since our arguments hold for any subsequence $(t_n)_{n\ge 1}$ it follows
$$
p_{t}(x,\cdot)\longrightarrow \frac{\rho dx}{\int_{\mathbb{R}}\rho(x)dx} 
$$
weakly as $t\to \infty$ for any $x\in \mathbb{R}$. Hence $(i)\Rightarrow (iv)$ and our proof is complete.
\end{proof}
\begin{rem}
Similarly to Remark \ref{conservativeglobalprop} one can see that under (\ref{semi1}) and (\ref{hbddvargamma}), properties (iii), (iv) of Theorem \ref{recurrence} and properties (ii), (iii) of Theorem \ref{recurrence2} are global assumptions and hence do not depend on the local behavior around the accumulation point.
\end{rem}

\begin{cor}\label{uniqueinvariantdist} 
Assume $((X_t)_{t\ge 0}, (P_x)_{x\in\mathbb{R}})$ is positive recurrent. 
Then $\frac{\rho dx}{\int_{\mathbb{R}}\rho(x)dx}$ is the unique invariant distribution.
\end{cor}
\begin{proof}
Let $\nu$ be an invariant distribution. Then using Theorem \ref{recurrence2}(iv), (\ref{U}), (\ref{K}) with $\rho dx$ replaced by $\nu$, and $\mu$ replaced by $\frac{\rho dx}{\int_{\mathbb{R}}\rho(x)dx}$,
we obtain similarly to the proof of $(i)\Rightarrow (iv)$ in Theorem \ref{recurrence2} that
$$
\nu(B)=\frac{\int_{B}\rho(x)dx}{\int_{\mathbb{R}}\rho(x)dx}
$$
for any Borel set $B$. The result hence follows.
\end{proof}

\subsection{Advection-diffusion in layered media}\label{33}
Let $(l_k)_{k\in\mathbb{Z}}, (r_k)_{k\in\mathbb{Z}}\subset \mathbb{R}$ be as at the beginning of section \ref{two}.
For $\alpha\in (0,1)$ consider the sequences
\begin{eqnarray*}
\gamma_{k+1}:=c_{\alpha}\sqrt{D_k}, \ \ \overline{\gamma}_{k+1}:=\overline{c}_{\alpha}\sqrt{\overline{D}_k}, \ \ k\in \mathbb{Z},
\end{eqnarray*}
where $(D_k)_{k\in\mathbb{Z}}, (\overline{D}_k)_{k\in\mathbb{Z}}\subset (0,\infty)$ and $c_{\alpha}, \overline{c}_{\alpha}>0$ are some constants that will be stated precisely below. We suppose that (\ref{semi1}), (\ref{hbddvargamma}), (\ref{global1}), and (\ref{global2}) hold. Then 
$$
\exists D:=\lim_{k\to\infty}D_k,\ \  \exists \overline{D}:=\lim_{k\to-\infty}\overline{D}_k \   \text{ and }\ D,\overline{D}>0. 
$$ 
Let 
\begin{eqnarray*}
\alpha_k:=\frac{\sqrt{D_k}}{\sqrt{D_k}+\sqrt{D_{k-1}}}, \ \ \overline{\alpha}_k:=\frac{\sqrt{\overline{D}_k}}{\sqrt{\overline{D}_k}+\sqrt{\overline{D}_{k-1}}}, \ \ k\in \mathbb{Z},
\end{eqnarray*}
and define 
$$
c_{\alpha}:=\frac{\alpha}{\sqrt{D}}, \ \ \overline{c}_{\alpha}:=\frac{1-\alpha}{\sqrt{\overline{D}}}.
$$
By Remark \ref{leGallweaker3}(ii), we know that (\ref{semi1}) together with (\ref{hbddvargamma}) are equivalent to (\ref{legallcond1}). Then, by Theorem \ref{legallweaker1} there exists 
a unique strong solution $Z^{\alpha}$ to 
\begin{eqnarray*}
Z^{\alpha}_t=x+W_t+ \sum_{k\in \mathbb{Z}}\left \{(2\alpha_k -1)\ell_t^{l_k}(Z^{\alpha})+(2\overline{\alpha}_k-1)\ell_t^{r_k}(Z^{\alpha})\right \}+(2\alpha-1)\ell_t^0(Z^{\alpha}),
\end{eqnarray*}
which is constructed with the help of the Dirichlet form that is determined by (\ref{gammadef}) and (\ref{closureDF}).\\
We now fix $\alpha\in (0,1)$. By (\ref{semi1}) and (\ref{hbddvargamma}), there exists $\Psi:\mathbb{R}\to \Psi(\mathbb{R})$ which is the difference of two convex functions, (continuous), piecewise linear with $\Psi(0)=0$ such that
\begin{equation*}
        \Psi'(x) = \begin{cases}
                       c_{\alpha}\sqrt{D_k} & \text{on $(l_k,l_{k+1})$}\\
                        \overline{c}_{\alpha}\sqrt{\overline{D}_k} & \text{on $(r_k,r_{k+1})$}.
                      \end{cases}
      \end{equation*}
Applying the It\^o-Tanaka-formula and formulas about local times from \cite{oukrut}, we obtain after a long calculation that $X:=\Psi(Z^{\alpha})$ 
is a strong solution to
\begin{eqnarray}\label{solutionX}
X_t=\Psi(x)+M_t+N_t,
\end{eqnarray}
where
$$
N_t=\sum_{k\in \mathbb{Z}}\left \{\frac{D_k-D_{k-1}}{D_{k}+D_{k-1}}\ell_t^{\Psi(l_k)}(X)+
\frac{\overline{D}_k-\overline{D}_{k-1}}{\overline{D}_{k}+\overline{D}_{k-1}}\ell_t^{\Psi(r_k)}(X)\right \}+\frac{\overline{D}-D}{\overline{D}+D}\ell_t^0(X)
$$
and 
$$
M_t=\int_0^t \sum_{k\in \mathbb{Z}}\left (c_{\alpha}\sqrt{D_k}\, 1_{[\Psi(l_k),\Psi(l_{k+1}))}+ \overline{c}_{\alpha}\sqrt{\overline{D}_k} \,1_{[\Psi(r_k),\Psi(r_{k+1}))}\right)(X_s)dW_s
$$
Note that we do not have to worry about the endpoints of the intervals $(\Psi(l_k),\Psi(l_{k+1}))$ and $(\Psi(r_k),\Psi(r_{k+1}))$ by  (\ref{n3}).
Define
$$
\sigma_1(x):= \sum_{k\in \mathbb{Z}}\left (c_{\alpha}\sqrt{D_k}\, 1_{[\Psi(l_k),\Psi(l_{k+1}))}+ \overline{c}_{\alpha}\sqrt{\overline{D}_k} \,1_{[\Psi(r_k),\Psi(r_{k+1}))}\right).
$$
Then $\sigma_1^2$ is locally uniformly strictly elliptic and so by results of \cite{fot}, we have that 
\begin{eqnarray}\label{solutionXDF}
{\cal A}^0(f,g) := \frac{1}{2}\int_{\Psi(\mathbb{R})} \sigma_1^2 f'g' \,dx, \ \ \  f,g\in C_0^{\infty}(\Psi(\mathbb{R}))
\end{eqnarray}
is closable in $L^2(\Psi(\mathbb{R}); dx)$. Denote the closure by $({\cal A}^0, D({\cal A}^0))$. Following the lines of arguments in this article (as for the Dirichlet form defined through 
(\ref{gammadef}), (\ref{closureDF})) one can verify that the unique solution $X$ to (\ref{solutionX}) is associated to the regular Dirichlet form
$({\cal A}^0, D({\cal A}^0))$. Let $(L^{{\cal A}^0},D(L^{{\cal A}^0}))$ be its generator.\\
In item (ii) of the following remark we point out a minor inconsistency in \cite{ram}. It can however easily be spotted.
\begin{rem}\label{inconsistencies}
(i) The state space of $X$ is $\Psi(\mathbb{R})$ and might be different from $\mathbb{R}$ if $\Psi$ has a low growth rate when approaching to 
$\pm \infty$. As an example consider the case where $r_k=k$ and $\sqrt{\overline{D}_k}=\frac{1}{k^2}$ for $k\ge 1$. \\
(ii) An invariant measure for $X$, is the Lebesgue measure restricted to $\Psi(\mathbb{R})$. This is directly visible from the definition of the corresponding Dirichlet form in (\ref{solutionXDF}). It will be finite, if and only if $\Psi(\mathbb{R})$ is bounded. 
It is evident from (\ref{solutionXDF}) that a scale function $\overline{h}$ for $X$, i.e. a piecewise linear function $\overline{h}:\Psi(\mathbb{R})\to \mathbb{R}$ with $\overline{h}(0)=0$ and ${\cal A}^0(\overline{h},g)=0$ for all $g$ is given by defining its derivatives as
\begin{equation*}
        \overline{h}'(x) = \begin{cases}
                       \frac{1}{c_{\alpha}^2 D_k} & \text{on $(\Psi(l_k),\Psi(l_{k+1}))$}\\
                       \frac{1}{\overline{c}_{\alpha}^2\overline{D}_k} & \text{on $(\Psi(r_k),\Psi(r_{k+1}))$}.
                      \end{cases}
      \end{equation*}
Then exactly as in Theorem \ref{recurrence}, we can see that $X$ is recurrent, if and only if
\begin{eqnarray}\label{recX}
\sum_{k\le 0}\frac{\Psi(l_{k+1})-\Psi(l_k)}{D_k}=\sum_{k\ge 0}\frac{\Psi(r_{k+1})-\Psi(r_k)}{\overline{D}_k}=\infty,
\end{eqnarray}
that is $-\overline{h}(\Psi(-\infty))=\infty=\overline{h}(\Psi(\infty))$. Noting that the speed measure of (\ref{solutionX}) 
is 
$$
\frac{2dx}{\sigma_1^2(x)\overline{h}'(x)}=2dx
$$ 
exactly as in Theorem \ref{recurrence2}, we can see that $X$ is
positive recurrent, if additionally to (\ref{recX})
\begin{eqnarray}\label{rec2X}
\sum_{k\le 0}(\Psi(l_{k+1})-\Psi(l_k))+\sum_{k\ge 0}(\Psi(r_{k+1})-\Psi(r_k))<\infty.
\end{eqnarray}
If (\ref{recX}) and (\ref{rec2X}) hold, then the normalized Lebesgue measure, i.e. the uniform distribution on (the bounded set !) $\Psi(\mathbb{R})$, is the unique invariant probability measure for $X$. 
This statement can be shown analogously to Corollary \ref{uniqueinvariantdist}. In conclusion, in case of an unbounded domain $\Psi(\mathbb{R})$ there is no invariant probability measure possible as stated 
in \cite[Remark 3.2]{ram}. But if $\Psi(\mathbb{R})$ is bounded the normalized Lebesgue measure is the unique invariant distribution 
as opposed to the statement of its non-existence in \cite[Remark 3.2]{ram}.\\
\end{rem}
Take an independent copy $(B_t)_{t\ge 0}$ of  $(W_t)_{t\ge 0}$ and let  $\sigma_2, {\beta}_2:\Psi(\mathbb{R})\to \mathbb{R}$ be locally bounded Borel-measurable functions. Assume additionally that 
$\sigma_2$ is strictly positive and locally uniformly bounded away from zero.
For $y\in \mathbb{R}$ consider the It\^o-process 
\begin{eqnarray*}
Y^y_t:=y+\int_0^t \sigma_2(X_s)dB_s +\int_0^t \beta_2(X_s)ds, \ \ t\ge 0.
\end{eqnarray*}
Clearly, $Y:=(Y^y)_{y\in \mathbb{R}}$ is non-explosive, since the paths of $X$ are continuous. \\
Below, we will show that $(X,Y)$ is associated to a generalized Dirichlet form, stationary and \lq\lq reversible\rq\rq\ 
(see Remark \ref{advection} and in particular \cite{Tr9}) with respect to 
the two dimensional Lebesgue measure $dxdy$. For this, we need some preparations.\\
Let
\begin{equation}
A=\begin{pmatrix}
\sigma_1^2 & 0\\
0& \sigma_2^2
\end{pmatrix}.
\end{equation}
Let $\partial_x f(x,y)$ denote the partial derivative in the $x$-coordinate and $\partial_y f(x,y)$ denote the partial derivative in the $y$-coordinate. 
Since $A$ is locally uniformly strictly elliptic it is well-known, that
\begin{eqnarray}
{\cal E}^0(f,g) &:= &\frac{1}{2}\int_{\Psi(\mathbb{R})}\int_{\mathbb{R}}\langle A\nabla f, \nabla g\rangle dx dy \nonumber \\
& = & \frac{1}{2}\int_{\Psi(\mathbb{R})}\int_{\mathbb{R}}\sigma_1^2\partial_{x}f\partial_x g \,dxdy +\frac{1}{2}\int_{\Psi(\mathbb{R})}\int_{\mathbb{R}}\sigma_2^2\partial_{y}f\partial_y g \,dxdy, 
\end{eqnarray}
with $f,g\in C_0^{\infty}(\Psi(\mathbb{R})\times\mathbb{R})$ is closable in $L^2(\Psi(\mathbb{R})\times\mathbb{R}; dxdy)$ and that the closure $({\cal E}^0, D({\cal E}^0))$ is a regular 
symmetric Dirichlet form. Let $(L^0,D(L^0))$ be the corresponding generator and for ${\cal{D}}\subset L^2(\Psi(\mathbb{R})\times\mathbb{R}; dxdy)$ set 
${\cal{D}}_b:={\cal{D}}\cap L^{\infty}(\Psi(\mathbb{R})\times\mathbb{R}; dxdy)$ and ${\cal{D}}_{0,b}:={\cal{D}}_b\cap \{f\text{ has compact support in } \Psi(\mathbb{R})\times\mathbb{R}  \}$. \\
Since $\beta_2$ only depends on the $x$-coordinate, the vector field $\beta=(\beta_1,\beta_2):\Psi(\mathbb{R})\times\mathbb{R}\to \mathbb{R}^2$, with $\beta_1\equiv 0$ is divergence free with respect to $dxdy$, i.e.
\begin{eqnarray}\label{divfrei}
\int_{\Psi(\mathbb{R})}\int_{\mathbb{R}}\langle \beta, \nabla f\rangle \, dx dy = \int_{\Psi(\mathbb{R})}\int_{\mathbb{R}}\beta_2(x)\partial_{y}f \,dxdy=0, \ \ \ 
\forall f\in C_0^{\infty}(\Psi(\mathbb{R})\times\mathbb{R}). 
\end{eqnarray}
Clearly, (\ref{divfrei}) extends to all $f$ in $D({\cal E}^0)_{0,b}$. Let 
$$
U_n:=(\Psi(l_{-n}), \Psi(r_{n}))\times (-n,n),\ \  n\ge 1.
$$
Then the $U_n$ are relatively compact open subsets of $\Psi(\mathbb{R})\times\mathbb{R}$ and we can consider the part 
Dirichlet forms $({\cal E}^{0,U_n} D({\cal E}^{0, U_n}))$, $n\ge 1$, as given in \cite[Theorem 4.4.3]{fot}.
Let $(L^{0, U_n},D(L^{0, U_n}))$ be the corresponding generators. Furthermore, since  $D({\cal E}^{0, U_n})_b \subset D({\cal E}^0)_{0,b}$, (\ref{divfrei}) 
holds for all $f\in  D({\cal E}^{0, U_n})_b$. Following the line of arguments in \cite{St1}, there exists for each $n\ge 1$, a closed 
extension $(\overline{L}^{U_n},D(\overline{L}^{U_n}))$ on $L^1(U_n;dxdy)$ of
\begin{eqnarray}\label{closedextUn}
L^{U_n}u:=L^{0, U_n}u+\beta_2\partial_y u,\ \ \  u\in  D(L^{0, U_n})_b
\end{eqnarray}
that generates a submarkovian $C_0$-semigroup of contractions on $L^1(U_n;dxdy)$. The part $(L^{U_n},D(L^{U_n}))$ of $(\overline{L}^{U_n},D(\overline{L}^{U_n}))$ on $L^2(U_n;dxdy)$ is then associated to a generalized Dirichlet form (cf. \cite[Section 1a)]{St1}). Then using a localization procedure by following \cite[Section 1b), Theorem 1.5]{St1} one can show, that there exists a closed 
extension $(\overline{L},D(\overline{L}))$ on $L^1(\Psi(\mathbb{R})\times\mathbb{R};dxdy)$ of
\begin{eqnarray}\label{closedext}
Lu:=L^{0}u+\beta_2\partial_y u,\ \ \  u\in  D(L^{0})_{0,b},
\end{eqnarray}
that generates a submarkovian $C_0$-semigroup of contractions on $L^1(\Psi(\mathbb{R})\times\mathbb{R};dxdy)$ and whose resolvent can be approximated by the resolvents of $(\overline{L}^{U_n},D(\overline{L}^{U_n}))$, $n\ge 1$. Note that for this one has to verify that $D(L^{0})_{0,b}\subset L^2(\Psi(\mathbb{R})\times\mathbb{R}; dxdy)$ densely, which holds since $D(L^{0, U_n})_b\subset D(L^{0})_{0,b}$  for any $n$, hence $D(L^{0})_{0,b}$ is dense in 
$L^2(U_n; dxdy)$ for any $n$. Then again analogously to the line of arguments in  \cite{St1}), one shows that the part $(L,D(L))$ of $(\overline{L},D(\overline{L}))$ on $L^2(\Psi(\mathbb{R})\times\mathbb{R};dxdy)$ is associated to a quasi-regular generalized Dirichlet form that has a nice additional structure which is known as condition D3 (for this we refer the interested reader to \cite[Section 4]{Tr4} and references therein). The identification of the associated process can then be performed similarly to what is done in \cite{Tr4} and here.
\begin{rem}\label{advection}
The process $(X,Y)$ has been constructed in \cite{ram} under the stronger additional assumptions that $\sigma_1, \sigma_2$, and  $\frac{\beta_2}{\sigma_2}$ are bounded. In particular $\beta_2$ is then also bounded. The components $X$, (resp. $Y$), represent the transversal, (resp. longitudinal) directions of advection-diffusion in layered media and the fundamental solution corresponding to the underlying Kolmogorov operator serves as a model for the concentration of a solute undergoing advection-diffusion there (see \cite{ram2}, \cite[section 3]{ram}). Having constructed $(X,Y)$ in a more general setting, one can study the asymptotic properties of $X$ and $Y$ as in \cite{ram} and \cite{ram2}. Our generalization for $X$ may be interpreted as increased heterogeneity of the layered media. The weaker assumptions on $\beta_2$ allow higher speed of transportation (advection) of the solute particles in the respective layers. We have seen in Remark \ref{inconsistencies}(ii) that $X$ has a unique invariant distribution if (\ref{recX}) and (\ref{rec2X}) hold. So in this case one may hope to obtain a central limit theorem 
for $(X,Y)$ as in \cite{ram2} (cf. also \cite[Remark 3.2]{ram}), i.e. one may hope to solve the Taylor-Aris problem. Heterogeneous dispersion in a longitudinal flow was carried out in \cite[section 2, 3]{ram2} with respect to a compact domain $G$ with finitely many layers and normal reflecting boundary condition at $\partial G$. In this regard, it could also be interesting to investigate the effect of replacing $G$ with the bounded domain $\Psi(\mathbb{R})$, 
thus allowing increased heterogeneity and inaccessible boundaries. Finally, we note that the (non-sectorial) Lyons-Zheng decomposition holds for $(X,Y)$ (see \cite{Tr9}). Hence one can use it as an additional tool to derive ergodic properties like for instance in \cite{Ta89}.
\end{rem}
{\bf Acknowledgment}: We would like to thank the anonymous referees, especially the second referee, for their valuable suggestions which helped to improve this manuscript significantly.

Youssef Ouknine\\ 
Department of Mathematics\\
Faculty of Sciences Semlalia\\
Cadi Ayyad University\\
Marrakesh, Morocco\\
E-mail: ouknine@uca.ma\\ \\
Francesco Russo\\
Ecole Nationale Sup\'erieure des Techniques Avanc\'ees, ENSTA-ParisTech\\
Unit\'e de Math\'ematiques appliqu\'ees,\\
828, boulevard des Mar\'echaux, \\
F-91120 Palaiseau, France \\
E-mail: francesco.russo@ensta-paristech.fr \\ \\ 
Gerald Trutnau\\
Department of Mathematical Sciences and \\
Research Institute of Mathematics of Seoul National University,\\
San56-1 Shinrim-dong Kwanak-gu, \\
Seoul 151-747, South Korea,  \\
E-mail: trutnau@snu.ac.kr

\end{document}